\documentclass[reqno]{amsart}

\usepackage[utf8]{inputenc}
\usepackage[usenames,dvipsnames]{color}
\usepackage{amsmath}
\usepackage{enumerate}
\usepackage{amssymb}
\usepackage{amsthm}
\usepackage{mathtools}
\usepackage[mathscr]{euscript}
\usepackage{tikz-cd}
\usepackage[hidelinks]{hyperref}
\usepackage{cleveref}

\usepackage{etoolbox}

\numberwithin{equation}{section}
\newtheorem{thm}[equation]{Theorem}
\newtheorem*{thm*}{Theorem}
\newtheorem*{prop*}{Proposition}
\newtheorem{prop}[equation]{Proposition}
\newtheorem{lem}[equation]{Lemma}
\newtheorem{cor}[equation]{Corollary}

\theoremstyle{remark}
\newtheorem{defn}[equation]{Definition}
\newtheorem*{defn*}{Definition}
\newtheorem*{setup*}{Setup}
\newtheorem{hyp}[equation]{Hypothesis}
\newtheorem*{hyp*}{Hypothesis}

\newtheorem{nota}[equation]{Notation}
\newtheorem*{nota*}{Notation}
\newtheorem{ex}[equation]{Example}

\newtheorem{rem}[equation]{Remark}

\newtheorem*{Ack}{Acknowledgements}

\newcommand{\nc}{\newcommand}
\nc{\dmo}{\DeclareMathOperator}

\nc{\sfA}{\mathsf{A}}
\nc{\sfB}{\mathsf{B}}
\nc{\sfC}{\mathsf{C}}
\nc{\sfD}{\mathsf{D}}
\nc{\sfE}{\mathsf{E}}
\nc{\sfF}{\mathsf{F}}
\nc{\sfG}{\mathsf{G}}
\nc{\sfH}{\mathsf{H}}
\nc{\sfI}{\mathsf{I}}
\nc{\sfJ}{\mathsf{J}}
\nc{\sfK}{\mathsf{K}}
\nc{\sfL}{\mathsf{L}}
\nc{\sfM}{\mathsf{M}}
\nc{\sfN}{\mathsf{N}}
\nc{\sfO}{\mathsf{O}}
\nc{\sfP}{\mathsf{P}}
\nc{\sfQ}{\mathsf{Q}}
\nc{\sfR}{\mathsf{R}}
\nc{\sfS}{\mathsf{S}}
\nc{\sfT}{\mathsf{T}}
\nc{\sfU}{\mathsf{U}}
\nc{\sfV}{\mathsf{V}}
\nc{\sfW}{\mathsf{W}}
\nc{\sfX}{\mathsf{X}}
\nc{\sfY}{\mathsf{Y}}
\nc{\sfZ}{\mathsf{Z}}

\nc{\scA}{\mathscr{A}}
\nc{\scB}{\mathscr{B}}
\nc{\scC}{\mathscr{C}}
\nc{\scD}{\mathscr{D}}
\nc{\scE}{\mathscr{E}}
\nc{\scF}{\mathscr{F}}
\nc{\scG}{\mathscr{G}}
\nc{\scH}{\mathscr{H}}
\nc{\scI}{\mathscr{I}}
\nc{\scJ}{\mathscr{J}}
\nc{\scK}{\mathscr{K}}
\nc{\scL}{\mathscr{L}}
\nc{\scM}{\mathscr{M}}
\nc{\scN}{\mathscr{N}}
\nc{\scO}{\mathscr{O}}
\nc{\scP}{\mathscr{P}}
\nc{\scQ}{\mathscr{Q}}
\nc{\scR}{\mathscr{R}}
\nc{\scS}{\mathscr{S}}
\nc{\scT}{\mathscr{T}}
\nc{\scU}{\mathscr{U}}
\nc{\scV}{\mathscr{V}}
\nc{\scW}{\mathscr{W}}
\nc{\scX}{\mathscr{X}}
\nc{\scY}{\mathscr{Y}}
\nc{\scZ}{\mathscr{Z}}

\nc{\mcA}{\mathcal{A}}
\nc{\mcB}{\mathcal{B}}
\nc{\mcC}{\mathcal{C}}
\nc{\mcD}{\mathcal{D}}
\nc{\mcE}{\mathcal{E}}
\nc{\mcF}{\mathcal{F}}
\nc{\mcG}{\mathcal{G}}
\nc{\mcH}{\mathcal{H}}
\nc{\mcI}{\mathcal{I}}
\nc{\mcJ}{\mathcal{J}}
\nc{\mcK}{\mathcal{K}}
\nc{\mcL}{\mathcal{L}}
\nc{\mcM}{\mathcal{M}}
\nc{\mcN}{\mathcal{N}}
\nc{\mcO}{\mathcal{O}}
\nc{\mcP}{\mathcal{P}}
\nc{\mcQ}{\mathcal{Q}}
\nc{\mcR}{\mathcal{R}}
\nc{\mcS}{\mathcal{S}}
\nc{\mcT}{\mathcal{T}}
\nc{\mcU}{\mathcal{U}}
\nc{\mcV}{\mathcal{V}}
\nc{\mcW}{\mathcal{W}}
\nc{\mcX}{\mathcal{X}}
\nc{\mcY}{\mathcal{Y}}
\nc{\mcZ}{\mathcal{Z}}

\nc{\mfp}{\mathfrak{p}}
\nc{\mfq}{\mathfrak{q}}
\nc{\mfm}{\mathfrak{m}}
\nc{\mfj}{\mathfrak{j}}
\nc{\mfs}{\mathfrak{s}}
\nc{\mfh}{\mathfrak{h}}

\nc{\rmh}{\mathrm{h}}
\nc{\rmfp}{\mathrm{fp}}
\nc{\rmc}{\mathrm{c}}
\nc{\rms}{\mathrm{s}}
\nc{\rma}{\mathrm{a}}
\nc{\rmb}{\mathrm{b}}
\nc{\rml}{\mathrm{l}}
\nc{\rmL}{\mathrm{L}}
\nc{\rmC}{\mathrm{C}}
\nc{\rmD}{\mathrm{D}}
\nc{\rmK}{\mathrm{K}}
\nc{\rmS}{\mathrm{S}}
\nc{\rmop}{\mathrm{op}}
\nc{\Ab}{\mathrm{Ab}}

\nc{\bfA}{\mathbf{A}}
\nc{\bfB}{\mathbf{B}}
\nc{\bfC}{\mathbf{C}}
\nc{\bfD}{\mathbf{D}}
\nc{\bfE}{\mathbf{E}}
\nc{\bfF}{\mathbf{F}}
\nc{\bfG}{\mathbf{G}}
\nc{\bfH}{\mathbf{H}}
\nc{\bfI}{\mathbf{I}}
\nc{\bfJ}{\mathbf{J}}
\nc{\bfK}{\mathbf{K}}
\nc{\bfL}{\mathbf{L}}
\nc{\bfM}{\mathbf{M}}
\nc{\bfN}{\mathbf{N}}
\nc{\bfO}{\mathbf{O}}
\nc{\bfP}{\mathbf{P}}
\nc{\bfQ}{\mathbf{Q}}
\nc{\bfR}{\mathbf{R}}
\nc{\bfS}{\mathbf{S}}
\nc{\bfT}{\mathbf{T}}
\nc{\bfU}{\mathbf{U}}
\nc{\bfV}{\mathbf{V}}
\nc{\bfW}{\mathbf{W}}
\nc{\bfX}{\mathbf{X}}
\nc{\bfY}{\mathbf{Y}}
\nc{\bfZ}{\mathbf{Z}}

\nc{\bbZ}{\mathbb{Z}}
\nc{\bbQ}{\mathbb{Q}}
\nc{\bbZp}{\mathbb{Z}_{(p)}}
\nc{\pruferp}{\mathbb{Z}(p^\infty)}
\nc{\pruferq}{\mathbb{Z}(q^\infty)}

\nc{\gG}{\Gamma}
\nc{\gL}{\Lambda}
\nc{\gD}{\Delta}
\nc{\gS}{\Sigma}

\nc{\ga}{\alpha}
\nc{\gb}{\beta}
\nc{\g}{\gamma}
\nc{\gd}{\delta}
\nc{\e}{\epsilon}
\nc{\gz}{\zeta}
\nc{\gh}{\eta}
\nc{\gu}{\theta}
\nc{\gi}{\iota}
\nc{\gk}{\kappa}
\nc{\gl}{\lambda}
\nc{\gm}{\mu}
\nc{\gn}{\nu}
\nc{\gj}{\xi}
\nc{\gp}{\pi}
\nc{\gr}{\rho}
\nc{\gs}{\sigma}
\nc{\gt}{\tau}
\nc{\gf}{\phi}
\nc{\gx}{\chi}
\nc{\gc}{\psi}
\nc{\go}{\omega}

\nc{\wh}{\widehat}
\nc{\wt}{\widetilde}
\nc{\ol}{\overline}
\nc{\ul}{\underline}
\nc{\tl}{\tilde}
\nc{\ot}{\otimes}
\nc{\xr}{\xrightarrow}

\nc{\ie}{\sl i.e.,}

\dmo{\Coloch}{Coloc^{hom}}
\dmo{\coloch}{coloc^{hom}}
\dmo{\Coloca}{Coloc^\ast}
\dmo{\coloca}{coloc^\ast}
\dmo{\Coloc}{Coloc}
\dmo{\coloc}{coloc}
\dmo{\Loct}{Loc^\ot}
\dmo{\loct}{loc^\ot}
\dmo{\Loca}{Loc^\ast}
\dmo{\loca}{loc^\ast}
\dmo{\Loc}{Loc}
\dmo{\loc}{loc}
\dmo{\Sm}{S^\ot}
\dmo{\Sma}{S^\ast}
\dmo{\Thick}{Thick^\ot}
\dmo{\thick}{thick^\ot}
\dmo{\Serre}{Serre^\ot}
\dmo{\Thom}{Thom}

\nc{\Modc}{\Mod(\scT^\rmc)}
\nc{\smodc}{\mathrm{mod}(\scT^\rmc)}
\nc{\tcop}{{\{\scT^\rmc\}}^{\mathrm{op}}}

\dmo{\cone}{cone}

\dmo{\Spec}{Spec}
\nc{\Spc}{\mathrm{Spc}(\scT^\rmc)}
\dmo{\Spcsm}{Spc^s}
\nc{\Spcs}{\Spcsm(\scT)}
\nc{\Spcsa}[1]{\mathrm{Spc}^\rms(#1)}
\nc{\Spcsl}[1]{\mathrm{Spc}^\rms(\scT/#1)}
\dmo{\Spch}{Spc^h(\scT^c)}
\dmo{\SPC}{SPC(\scT)}
\dmo{\Supps}{Supp^s}
\dmo{\Supph}{Supp^h}
\dmo{\SUPP}{SUPP}
\dmo{\Supp}{Supp}
\dmo{\Cosupp}{Cosupp}
\dmo{\Cosupph}{Cosupph}
\dmo{\cosupp}{cosupp}
\dmo{\cosupphom}{cosupph}
\dmo{\supps}{supp^s}
\dmo{\supph}{supp^h}
\dmo{\supp}{supp}
\dmo{\supphom}{supph}
\dmo{\cosupps}{cosupp^s}
\dmo{\Cosupps}{Cosupp^s}

\dmo{\sg}{s_\Gamma}
\dmo{\sga}{s_\Gamma^\ast}
\dmo{\tausg}{\tau_{s_\Gamma}}
\dmo{\tausga}{\tau_{s_\Gamma^\ast}}
\dmo{\sigmasg}{\sigma_{s_\Gamma}}
\dmo{\sigmasga}{\sigma_{s_\Gamma^\ast}}
\dmo{\cg}{c_\Gamma}
\dmo{\cga}{c_\Gamma^\ast}
\dmo{\taucg}{\tau_{c_\Gamma}}
\dmo{\taucga}{\tau_{c_\Gamma^\ast}}
\dmo{\sigmacg}{\sigma_{c_\Gamma}}
\dmo{\sigmacga}{\sigma_{c_\Gamma^\ast}}

\dmo{\Mod}{Mod}
\dmo{\smod}{mod}
\dmo{\Modu}{\ul{\Mod}}
\dmo{\smodu}{\ul{\smod}}
\dmo{\Proj}{Proj}
\dmo{\proj}{proj}
\dmo{\Flat}{Flat}
\dmo{\Inj}{Inj}
\dmo{\PInj}{PInj}

\dmo{\Cohs}{Coh}
\dmo{\QCoh}{QCoh}

\dmo{\RHom}{RHom}
\dmo{\Hom}{Hom}
\dmo{\shom}{hom}
\dmo{\Ext}{Ext}
\dmo{\Tor}{Tor}
\dmo{\End}{End}
\dmo{\Aut}{Aut}
\dmo{\Ob}{Ob}
\dmo{\Mor}{Mor}
\dmo{\Ph}{Ph}
\dmo{\Ann}{Ann}

\dmo{\Ker}{Ker}
\dmo{\coker}{Coker}
\dmo{\im}{Im}
\dmo{\colim}{colim}
\dmo{\hocolim}{Hocolim}
\dmo{\Id}{Id}

\nc{\Char}[1]{{\color{ForestGreen}#1}}

\tikzcdset{scale cd/.style={every label/.append style={scale=#1},
    cells={nodes={scale=#1}}}}
    
\tikzset{rot270/.style={anchor=south, rotate=270, inner sep=1.0mm}}

\nc{\rcolon}{\nobreak \mskip 6muplus1mu\mathpunct {}\nonscript \mkern -\thinmuskip {:}\mskip 2mu\relax} 

\nc{\set}[2]{\big\{\, #1 \ \big| \ #2 \,\big\}} 

\nc{\paren}[2]{\! \big(\, #1 \ \big| \ #2 \big)} 
\nc{\parens}[2]{\! \big(\, #1 \ \big| \ #2 \,\big)} 

\nc{\qquadtext}[1]{\qquad\textrm{#1}\qquad} 

\makeatletter
\patchcmd{\@setaddresses}{\indent}{\noindent}{}{}
\patchcmd{\@setaddresses}{\indent}{\noindent}{}{}
\patchcmd{\@setaddresses}{\indent}{\noindent}{}{}
\patchcmd{\@setaddresses}{\indent}{\noindent}{}{}
\makeatother

\begin{document}
\title{Costratification and actions of tensor-triangulated categories}
\author{Charalampos Verasdanis}
\address{Charalampos Verasdanis \\ School of Mathematics and Statistics \\ University of Glasgow}
\email{c.verasdanis.1@research.gla.ac.uk}
\date{}

\begin{abstract}
We develop the theory of costratification in the setting of relative tensor-triangular geometry, in the sense of Stevenson, providing a unified approach to classification results of Neeman and Benson--Iyengar--Krause, while laying the foundations for future applications. In addition, we introduce and study prime localizing submodules and prime colocalizing $\shom$-submodules, in the first case, generalizing objectwise-prime localizing tensor-ideals. We apply our results to show that the derived category of quasi-coherent sheaves over a noetherian separated scheme is costratified.
\end{abstract}

\maketitle

\vspace{-0.41em}

\vskip-\baselineskip\vskip-\baselineskip
\tableofcontents

\vskip-\baselineskip\vskip-\baselineskip
\vskip-\baselineskip\vskip-\baselineskip

\section{Introduction}

The theory of cosupport and costratification in tensor-triangulated categories was initiated by Bensor--Iyengar--Krause~\cite{BensonIyengarKrause12}, inspired by the classification of colocalizing subcategories of the derived category of a commutative noetherian ring by Neeman~\cite{Neeman11}. Their main application was the classification of $\Hom$-closed colocalizing subcategories of the stable module category of a finite group. Compared to the theory of support and stratification~\cite{BensonIyengarKrause08,BensonIyengarKrause11a,BensonIyengarKrause11b,BarthelHeardSanders23} (which is related to the classification of localizing subcategories) the theory of costratification has not been explored as much. In particular, a theorem that unifies the aforementioned classifications has not yet been stated and the machinery of~\cite{BensonIyengarKrause12} depends on the action of a commutative noetherian ring on the triangulated category involved. The aim of this paper is to provide such a theorem and develop the theory of costratification in the context of relative tensor-triangular geometry~\cite{Stevenson13}. The advantage to this approach is that it allows our theory to apply in cases of triangulated categories that are not necessarily tensor-triangulated but are endowed with an action of a tensor-triangulated category. For instance, singularity categories of commutative noetherian rings~\cite{Krause05,Stevenson14b}. This will be the subject of future work. It should be noted that progress on the topic of costratification has also been made independently by Barthel--Castellana--Heard--Sanders, focusing on the Balmer--Favi support~\cite{BarthelCastellanaHeardSanders23}.

Let $\scT$ be a rigidly-compactly generated tensor-triangulated category and let $\scK$ be a compactly generated triangulated category endowed with an action of $\scT$; see~\Cref{sec:prelim} for details. In the special case where $\scK=\scT$, the action is given by the tensor product of $\scT$.

\Cref{sec:prelim} consists of preliminary material --- including a very brief account of the Balmer--Favi support~\cite{BalmerFavi11} and the smashing spectrum and the small smashing support~\cite{BalchinStevenson23} --- and establishes concepts and basic lemmas that we will be using throughout the paper.

In~\Cref{sec:costrat}, we introduce the notion of a good support--cosupport pair on $\scT$ taking values in a space $S$, which induces a support--cosupport pair on $\scK$. For example, one could take the small smashing support--cosupport or the Balmer--Favi support--cosupport or the BIK support--cosupport. In our first main result,~\Cref{thm:costrat-coltg-comin}, we prove that $\scK$ is costratified (meaning that there is a bijective correspondence between certain subsets of $S$ and the collection of colocalizing $\shom$-submodules of~$\scK$) if and only if $\scK$ satisfies two conditions: the colocal-to-global principle and cominimality. These two conditions are in a sense dual to the more well-established local-to-global principle and minimality, which were introduced in~\cite{BensonIyengarKrause11a} and studied further in~\cite{Stevenson13,Stevenson17,Stevenson18a,BarthelHeardSanders23}. Moreover, in~\Cref{cor:ltg-implies-coltg}, we prove that the local-to-global principle implies the colocal-to-global principle.

\Cref{sec:primes} is devoted to the study of prime localizing submodules and prime colocalizing $\shom$-submodules of $\scK$, the former specializing to the class of objectwise-prime localizing tensor-ideals when $\scK=\scT$; see~\cite{BalchinStevenson23,Verasdanis23}. We prove that if $\scK$ is costratified, then there is a bijective correspondence between points of a certain subspace of $S$ and prime colocalizing $\shom$-submodules of $\scK$, obtaining a complete description of the latter (\Cref{thm:costrat-hom-primes}). The correspondence is given by~\Cref{thm:costrat-coltg-comin}. In addition, we analyze the relation of prime localizing submodules and prime colocalizing $\shom$-submodules with the Action Formula (which generalizes the Tensor Product Formula) and the Internal-Hom Formula. 

In~\Cref{sec:smash}, we show that costratification of $\scK$ can be reduced to costratification of certain smashing localizations of $\scK$; see~\Cref{thm:comin-local}. For instance, if $\scK=\scT$ and the action of $\scT$ on itself is the tensor product of $\scT$ and the support--cosupport theory one fixes is given by the smashing spectrum $\Spcs$ and the small smashing support (assuming that $\Spcs$ is $T_D$) we have the following: Provided that $\scT$ satisfies the colocal-to-global principle, $\scT$ is costratified if and only if $\scT/P$ is costratified, for all $P\in \Spcs$; see~\Cref{cor:comin-local-smash}. As a direct consequence of~\Cref{cor:comin-local-smash}, we have: If $\Spcs=\bigcup_{j\in J}V_{\scS_j}$ is a cover of $\Spcs$ by closed subsets such that each $\scT/\scS_j$ satisfies cominimality, then (provided that $\scT$ satisfies the colocal-to-global principle) $\scT$ is costratified; see~\Cref{cor:comin-local-smash-covers}. In~\Cref{cor:comin-local-bf} and~\Cref{cor:comin-covers-compacts}, we state the analogous results for the Balmer spectrum $\Spc$ and the Balmer--Favi support (assuming that every point of $\Spc$ is visible). In~\Cref{thm:costrat-actions-covers}, we prove a generalization of~\Cref{cor:comin-covers-compacts} by replacing $\scT$ with a $\scT$-module $\scK$.

Finally, in~\Cref{sec:derived}, we present our application: Using the general machinery developed throughout, we first give in~\Cref{thm:affine-costrat} a more streamlined proof of Neeman's classification of the colocalizing subcategories of the derived category of a commutative noetherian ring and then we combine~\Cref{thm:affine-costrat} with~\Cref{cor:comin-covers-compacts} to show that the derived category of quasi-coherent sheaves over a noetherian separated scheme is costratified.
\begin{Ack}
I am grateful to Greg Stevenson for enlightening conversations and to Paul Balmer for his interest and suggestions that led to improvements.
\end{Ack}

\section{Preliminaries}\label{sec:prelim}
\subsection{Actions and basic lemmas}

Throughout, $\scT=(\scT,\ot,1)$ will be a rigidly-compactly generated tensor-triangulated category (big tt-category) as in~\cite{BalmerFavi11}, i.e., $\scT$ is a tensor-triangulated category with arbitrary coproducts and it is generated by its compact objects as a localizing subcategory. Moreover, the (essentially small) subcategory $\scT^\rmc$ of compact objects of $\scT$ is a tensor-triangulated subcategory and the rigid objects of $\scT$ coincide with the compact objects. The internal-hom functor of $\scT$ will be denoted by $[-,-]\colon \scT^\rmop \times \scT \to \scT$.

Let $\scK$ be a compactly generated triangulated category and let $\ast\colon\scT\times \scK\to \scK$ be an action of $\scT$ on $\scK$, in the sense of~\cite{Stevenson13}. In short, $\ast\colon \scT\times \scK\to \scK$ is a coproduct-preserving triangulated functor in each variable such that there exist natural (in all variables) isomorphisms $\ga_{X,Y,A}\colon X\ast(Y\ast A) \xr{\cong} (X\ot Y)\ast A$ and $l_A\colon 1\ast A \xr{\cong} A,\, \forall X,Y\in \scT,\, \forall A\in \scK$. The natural isomorphism $\ga$ is called the \emph{associator} and the natural isomorphism $l$ is called the \emph{unitor}. There is also a host of coherence conditions that need to be satisfied; we refer the reader to the aforementioned source for details. We call $\scK=(\scK,\ast)$ a $\scT$-\emph{module}.

By definition, for every object $X\in \scT$, the functor $X\ast -\colon \scK\to \scK$ is a coproduct-preserving triangulated functor. Hence, by Brown representability, $X\ast -$ admits a right adjoint $[X,-]_\ast \colon \scK \to \scK$. Assembling these right adjoints yields a functor $[-,-]_\ast\colon \scT^\rmop\times \scK \to \scK$ that we call the \emph{relative internal-hom}. Since $[1,-]_\ast$ is the right adjoint of $1\ast -\cong \Id_\scK$, it holds that $[1,-]_\ast\cong \Id_\scK$. Specifically, the composite $m\coloneqq \Id_\scK\to [1,1\ast -]_\ast \xr{[1,l]_\ast} [1,-]_\ast$, where the first map is the unit of adjunction, is a natural isomorphism (which we call the \emph{hom-unitor}).

\begin{lem}\label{lem:hom-associativity}
Let $X,Y\in \scT$ and $A\in \scK$. Then there exists a natural (in all variables) isomorphism $\gb_{X,Y,A}\colon [X\ot Y,A]_\ast \xr{\cong} [X,[Y,A]_\ast]_\ast$ called the hom-associator.
\end{lem}

\begin{proof}
Let $B\in\scK$. By the adjunction between the action and the relative internal-hom and the relation $(X\ot Y) \ast B \cong (Y\ot X)\ast B \cong Y\ast(X\ast B)$, we have:
\begin{align*}
\Hom_\scK(B,[X\ot Y,A]_\ast) &\cong \Hom_\scK((X\ot Y)\ast B,A)\\
&\cong \Hom_\scK((Y\ot X)\ast B,A)\\
&\cong \Hom_\scK(Y\ast(X\ast B),A)\\
& \cong \Hom_\scK(X\ast B,[Y,A]_\ast)\\
& \cong \Hom_\scK(B,[X,[Y,A]_\ast]_\ast).
\end{align*}
Consequently, for $B=[X\ot Y,A]_\ast$, the image of the identity morphism on $[X\ot Y,A]_\ast$ under the above series of isomorphisms gives a natural (in all variables) isomorphism $\gb_{X,Y,A}\colon [X\ot Y,A]_\ast \xr{\cong} [X,[Y,A]_\ast]_\ast$.
\end{proof}

\begin{nota}
Let $\scK_1$ and $\scK_2$ be two $\scT$-modules. For $i=1,2$, the relative internal-hom of $\scK_i$ will be denoted by $[-,-]_i$. Let $X,Y\in \scT$ and $A\in \scK_i$. The associator and unitor natural isomorphisms will be denoted by $\ga^i_{X,Y,A}$ and $l^i_A$, respectively. The hom-associator and hom-unitor natural isomorphisms will be denoted by $\gb^i_{X,Y,A}$ and $m^i_A$, respectively. The unit and the counit of the action-hom adjunction will be denoted by $u^i_{X,A}\colon A \to [X,X\ast_i A]_i$ and $c^i_{X,A}\colon X\ast_i [X,A]_i \to A$, respectively. We denote by $\gs_{X,Y}\colon X\ot Y \xr{\cong} Y\ot X$ the symmetry natural isomorphism. We denote by $c_{X,-}\colon X\ot [X,-] \to \Id_\scT$ the counit of the adjunction $X\ot - \dashv [X,-]$. Set $X^\vee \coloneqq [X,1]$ and define the morphism $\mathrm{ev}^i_{X,A} \colon X\ast_i A \to [X^\vee,A]_i$ as the following composite:
\begin{flalign*}
&X \ast_i A \xr{u^i_{X^\vee,X \ast_i A}} [X^\vee,X^\vee\ast_i (X \ast_i A)]_i \xr[\cong]{[X^\vee,\ga^i_{X^\vee,X,A}]_i} [X^\vee,(X^\vee \ot X)\ast_i A]_i&
\end{flalign*}
\[
\xr[\cong]{[X^\vee,\gs_{X^\vee,X}\ast_i A]_i} [X^\vee,(X\ot X^\vee)\ast_i A]_i \xr{[X^\vee,c_{X,1}\ast_i A]_i} [X^\vee,1\ast_i A]_i \xr[\cong]{[X^\vee,l^i_A]_i} [X^\vee,A]_i.
\]
\end{nota}

\begin{defn}\label{defn:action-pres}
A functor $F\colon \scK_1 \to \scK_2$, between $\scT$-modules $\scK_1$ and $\scK_2$, is called \emph{action-preserving} if there is a natural isomorphism $\phi\colon F(-\ast_1 -) \to - \ast_2 F(-)$ between the functors $F(-\ast_1 -),\, - \ast_2 F(-)\colon \scT\times \scK_1 \to \scK_2$ such that, for all $X,Y\in \scT$ and for all $A\in \scK_1$, the following diagrams commute:
\[
\begin{tikzcd}[column sep=4em,row sep=3.5em]
F(X\ast_1(Y\ast_1 A)) \rar["\phi_{X{,}Y\ast_1 A}"] \dar["F\ga^1_{X,Y,A}"'] & X\ast_2 F(Y\ast_1 A) \rar["X\ast_2 \phi_{Y,A}"] & X \ast_2 (Y\ast_2 FA) \dar["\ga^2_{X,Y,FA}"]
\\
F((X\ot Y)\ast_1 A) \arrow[rr,"\phi_{X\ot Y,A}"] && (X\ot Y) \ast_2 FA,
\end{tikzcd}
\]
\[
\begin{tikzcd}[row sep=3em]
F(1\ast_1 A) \rar["\phi_{1,A}"] \dar["Fl^1_A"']& 1\ast_2 FA \dlar["l^2_{FA}"]
\\
FA.
\end{tikzcd}
\]
\end{defn}

\begin{defn}\label{defn:hom-pres}
A functor $G\colon \scK_2 \to \scK_1$, between $\scT$-modules $\scK_2$ and $\scK_1$, is called \emph{$\shom$-preserving} if there is a natural isomorphism $\psi\colon [-,G(-)]_1 \to G[-,-]_2$ between the functors $[-,G(-)]_1,\, G[-,-]_2\colon \scT^\rmop\times \scK_2 \to \scK_1$ such that, for all $X,Y\in \scT$ and for all $B\in \scK_2$, the following diagrams commute:
\[
\begin{tikzcd}[column sep=4.3em,row sep=3.5em]
{[}X,{[}Y,GB{]}_1{]}_1 \rar["{[}X{,}\psi_{Y,B}{]}_1"] \dar["\gb^1_{X,Y,GB}"'] & {[}X,G{[}Y,B{]}_2{]}_1 \rar["\psi_{X,{[}Y{,}B{]}_2}"] & G{[}X,{[}Y,B{]}_2{]}_2 \dar["G\gb^2_{X,Y,B}"]
\\
{[}X\ot Y,GB{]}_1 \arrow[rr,"\psi_{X\ot Y{,}B}"] && G{[}X\ot Y,B{]}_2,
\end{tikzcd}
\]
\[
\begin{tikzcd}[row sep=3em]
GB \drar["Gm^2_B"] \dar["m^1_{GB}"']
\\
{[}1,GB{]}_1 \rar["\psi_{1,B}"'] & G{[}1,B{]}_2.
\end{tikzcd}
\]
\end{defn}

\begin{lem}\label{lem:restrictions-to-compacts}
Let $\scT$ and $\scK$ be triangulated categories with $\scT$ compactly generated and let $F_1,F_2\colon \scT \to \scK$ be coproduct-preserving triangulated functors (or contravariant triangulated functors that send coproducts to products). If there is a natural transformation $\theta\colon F_1 \to F_2$ such that $\theta_x$ is an isomorphism, for all $x\in \scT^\rmc$, then $\theta$ is a natural isomorphism.
\end{lem}

\begin{proof}
The subcategory $\scX=\set{X\in \scT}{\theta_X\colon F_1 X \to F_2 X \text{ is an isomorphism}}$ is a localizing subcategory of $\scT$ that contains $\scT^\rmc$. Consequently, $\scX=\scT$ and this proves the statement.
\end{proof}

\begin{lem}\label{lem:right-adj-action}
Let $F\colon \scK_1\to \scK_2$ be a coproduct and action-preserving triangulated functor between $\scT$-modules and let $G\colon \scK_2\to \scK_1$ be the right adjoint to $F$. Then $G$ is $\shom$-preserving. If $G$ is coproduct-preserving, then~$G$ is action-preserving. If $F$ is product-preserving, then $F$ is $\shom$-preserving.
\end{lem}

\begin{proof}
We denote by $\eta\colon \Id_{\scK_1} \to GF$ and $\varepsilon \colon FG \to \Id_{\scK_2}$ the unit and the counit, respectively, of the adjunction $F\dashv G$. Let $A\in \scK_1$, $B\in \scK_2$ and $X\in \scT$. Then
\begin{align*}
\Hom_{\scK_1}(A,[X,GB]_1)&\cong \Hom_{\scK_1}(X\ast A,GB)\\
&\cong \Hom_{\scK_2}(F(X\ast A),B)\\
&\cong \Hom_{\scK_2}(X\ast FA,B)\\
&\cong \Hom_{\scK_2}(FA,[X,B]_2)\\
&\cong \Hom_{\scK_1}(A,G[X,B]_2).
\end{align*}
For $A=[X,GB]_1$, the image of the identity morphism on $[X,GB]_1$ under the above series of isomorphisms provides a natural (in both variables) isomorphism $\psi_{X,B}\colon [X,GB]_1 \to G[X,B]_2$ that satisfies the conditions of~\Cref{defn:hom-pres}, showing that $G$ is $\shom$-preserving. More precisely, $\psi_{X,B}$ is the following composite:
\begin{flalign*}
&[X,GB]_1 \xr{\eta_{[X,GB]_1}}
GF[X,GB]_1 \xr{Gu^2_{X,F[X,GB]_1}}
G[X,X \ast_2 F[X,GB]_1]_2&
\end{flalign*}
\begin{flalign*}
&\xr{G[X,\phi^{-1}_{X,[X,GB]_1}]_2}
G[X,F(X \ast_1 [X,GB]_1)]_2
\xr{G[X, F(c^1_{X,GB})]_2}
G[X,FGB]_2&
\end{flalign*}
\begin{flalign*}
&\xr{G[X,\varepsilon_B]_2} G[X,B]_2.&
\end{flalign*}

Now suppose that $G$ preserves coproducts. We define a natural transformation $\xi_{X,B}\colon X\ast_1 GB \to G(X\ast_2 B)$ as the composite:
\[
X\ast_1 GB \xr{\eta_{X\ast_1 GB}} GF(X\ast_1 GB) \xr{G\phi_{X,GB}} G(X\ast_2 FGB) \xr{G(X\ast_2 \varepsilon_B)} G(X\ast_2 B).
\]
We claim that the square
\begin{equation}\label{cd:nt}
\begin{tikzcd}[row sep=3.5em,column sep=3em]
X \ast_1 GB \rar["\xi_{X,B}"] \dar["\mathrm{ev}^1_{X,GB}"']& G(X\ast_2 B) \dar["G\mathrm{ev}^2_{X,B}"]
\\
{[}X^\vee,GB{]}_1 \rar["\psi_{X^\vee,B}","\cong"'] & G{[}X^\vee,B{]}_2
\end{tikzcd}
\end{equation}
commutes. First, square~\eqref{cd:nt} can be expanded as follows:
\[
\begin{tikzcd}[row sep=3.5em,scale cd=0.755]
\arrow[r,phantom,"(1)",shift right=2.3em]
X \ast_1 GB \rar["\eta_{X\ast_1 GB}"] \dar["\mathrm{ev}^1_{X,GB}"']
& GF(X \ast_1 GB) \rar["G\phi_{X,GB}"] \dar["GF\mathrm{ev}^1_{X,GB}"']
& G(X\ast_2 FGB) \rar["G(X\ast_2 \varepsilon_B)"] \dar["G\mathrm{ev}^2_{X,FGB}"]
\arrow[r,phantom,"(3)",shift right=2.3em]
& G(X\ast_2 B) \dar["G\mathrm{ev}^2_{X,B}"]
\\
{[}X^\vee,GB{]}_1 \rar["\eta_{{[}X^\vee,GB{]}_1}"]
& GF{[}X^\vee,GB{]}_1 \dar["Gu^2_{X^\vee,F{[}X^\vee,GB{]}_1}"']
\arrow[r,phantom,"(2)"]
& G{[}X^\vee,FGB{]}_2 \rar["G{[}X^\vee{,}\varepsilon_B{]}_2"]
& G{[}X^\vee,B{]}_2
\\
& G{[}X^\vee,X^\vee \ast_2 F{[}X^\vee,GB{]}_1{]}_2
\rar["G{[}X^\vee{,}\phi^{-1}_{X^\vee,{[}X^\vee,GB{]}_1}{]}_2"',outer sep=5pt]
& G{[}X^\vee,F(X^\vee \ast_1 {[}X^\vee,GB{]}_1){]}_2
\uar["G{[}X^\vee{,} F(c^1_{X^\vee,GB}){]}_2"']
\end{tikzcd}
\]
where square $(1)$ commutes by naturality of $\eta$ and square $(3)$ commutes by naturality of $\mathrm{ev}^2_X$. Therefore, in order to show that~\eqref{cd:nt} commutes, it suffices to show that diagram $(2)$ commutes. We will prove this slightly more generally. We claim that the following diagram commutes:
\begin{equation}\label{diag:middle}
\begin{tikzcd}[row sep=3.5em]
& F(X \ast_1 A) \rar["\phi_{X,A}"] \dar["F\mathrm{ev}^1_{X,A}"']
& X\ast_2 FA \dar["\mathrm{ev}^2_{X,FA}"]
\\
& F[X^\vee, A]_1 \dar["u^2_{X^\vee,F[X^\vee,A]_1}"']
& {[X^\vee, FA]_2}
\\
& {[X^\vee,X^\vee \ast_2 F[X^\vee,A]_1]_2}
\rar["{[X^\vee,\phi^{-1}_{X^\vee,[X^\vee,A]_1}]_2}"',outer sep=5pt]
& {[X^\vee,F(X^\vee \ast_1 [X^\vee,A]_1)]_2}
\uar["{[X^\vee, F(c^1_{X^\vee,A})]_2}"'].
\end{tikzcd}
\end{equation}
Set 
\[
\begin{array}{ll}
f_1 = u^2_{X^\vee,F{[X^\vee,X^\vee\ast_1(X\ast_1 A)]_1}},
& g_1 = [X^\vee,\phi^{-1}_{X^\vee,{[}X^\vee,X^\vee \ast_1 (X\ast_1 A){]}_1}]_2,
\\
f_2 = u^2_{X^\vee,F{[}X^\vee,(X^\vee\ot X) \ast_1 A{]}_1},
& g_2 = [X^\vee,\phi^{-1}_{X^\vee,{[}X^\vee,(X^\vee\ot X)\ast_1 A{]}_1}]_2,
\\
f_3 = u^2_{X^\vee,F{[}X^\vee,(X\ot X^\vee) \ast_1 A{]}_1},
& g_3 = [X^\vee,\phi^{-1}_{X^\vee,{[}X^\vee,(X\ot X^\vee)\ast_1 A{]}_1}]_2,
\\
f_4 = u^2_{X^\vee,F{[}X^\vee,1\ast_1A{]}_1},
& g_4 = [X^\vee,\phi^{-1}_{X^\vee,{[}X^\vee,1\ast_1 A{]}_1}]_2,
\\
f_5 = u^2_{X^\vee,F[X^\vee,A]_1},
& g_5 = [X^\vee,\phi^{-1}_{X^\vee,[X^\vee,A]_1}]_2,
\end{array}
\]
\[
\begin{array}{l}
h_1 = [X^\vee,X^\vee \ast_2 \phi_{X,A}]_2 \circ [X^\vee,\phi_{X^\vee, X\ast_1 A}]_2 \circ [X^\vee,Fc^1_{X^\vee \ast_1 (X\ast_1 A)}]_2,\\
h_2 = [X^\vee,\phi_{X^\vee\ot X, A}]_2 \circ [X^\vee,Fc^1_{X^\vee, (X^\vee\ot X)\ast_1 A}]_2,\\
h_3 = [X^\vee,\phi_{X\ot X^\vee, A}]_2 \circ [X^\vee,Fc^1_{X^\vee, (X\ot X^\vee)\ast_1 A}]_2,\\
h_4 = [X^\vee,\phi_{1,A}]_2 \circ [X^\vee,Fc^1_{X^\vee,1\ast_1 A}]_2,\\
h_5 = [X^\vee, F(c^1_{X^\vee,A})]_2,
\end{array}
\]
and expand diagram~\eqref{diag:middle} as below:
\begin{flalign*}
&\begin{tikzcd}[scale cd=0.575,column sep=1em,row sep=3em,ampersand replacement=\&]
F(X \ast_1 A)
\dar["Fu^1_{X^\vee,X\ast_1 A}"'description]
\arrow[rrr,"\cong"',"\phi_{X,A}"]
\&\&\& X\ast_2 FA
\dar["u^2_{X^\vee,X\ast_2 FA}"description]
\\
\arrow[r,phantom,"(1)",shift right=2.3em]
F{[}X^\vee,X^\vee\ast_1(X\ast_1 A){]}_1
\dar["F{[X^\vee,\ga^1_{X^\vee,X,A}]_1}"'description]
\rar["f_1"]
\& {[}X^\vee,X^\vee \ast_2 F{[}X^\vee,X^\vee \ast_1 (X\ast_1 A){]}_1{]}_2
\dar["{[X^\vee,X^\vee\ast_2 F{[X^\vee,\ga^1_{X^\vee,X,A}]_1}]_2}"description]
\rar["g_1"]
\arrow[r,phantom,"(2)",shift right=2.3em]
\& {[}X^\vee,F(X^\vee \ast_1{[}X^\vee,X^\vee\ast_1 (X\ast_1 A){]}_1){]}_2
\dar["{[X^\vee,F(X^\vee \ast_1 [X^\vee,\ga^1_{X^\vee,X,A}]_1)]_2}"description]
\rar["h_1"]
\arrow[r,phantom,"(3)",shift right=2.3em]
\& {[}X^\vee,X^\vee\ast_2(X\ast_2 FA){]}_2
\dar["{[X^\vee,\ga^2_{X^\vee,X,FA}]_2}"description]
\\
F{[}X^\vee,(X^\vee\ot X) \ast_1 A{]}_1
\dar["F{[X^\vee,\gs_{X^\vee,X}\ast_1 A]_1}"'description]
\rar["f_2"]
\arrow[r,phantom,"(4)",shift right=2.3em]
\& {[}X^\vee,X^\vee \ast_2 F{[}X^\vee,(X^\vee\ot X)\ast_1 A{]}_1{]}_2
\dar["{[X^\vee,X^\vee\ast_2 F{[X^\vee,\gs_{X^\vee,X}\ast_1 A]_1}]_2}"description]
\rar["g_2"]
\arrow[r,phantom,"(5)",shift right=2.3em]
\& {[}X^\vee,F(X^\vee \ast_1{[}X^\vee,(X^\vee\ot X)\ast_1 A{]}_1){]}_2
\dar["{[X^\vee,F(X^\vee \ast_1 [X^\vee,\gs_{X^\vee,X}\ast_1 A]_1)]_2}"description]
\rar["h_2"]
\arrow[r,phantom,"(6)",shift right=2.3em]
\& {[}X^\vee,(X^\vee\ot X) \ast_2 FA{]}_2
\dar["{[X^\vee,\gs_{X^\vee,X}\ast_2 FA]_2}"description]
\\
F{[}X^\vee,(X\ot X^\vee) \ast_1 A{]}_1
\dar["F{[X^\vee,c_{X,1}\ast_1 A]_1}"'description]
\rar["f_3"]
\arrow[r,phantom,"(7)",shift right=2.3em]
\& {[}X^\vee,X^\vee \ast_2 F{[}X^\vee,(X\ot X^\vee)\ast_1 A{]}_1{]}_2
\dar["{[X^\vee,X^\vee\ast_2 F{[X^\vee,c_{X,1}\ast_1 A]_1}]_2}"description]
\rar["g_3"]
\arrow[r,phantom,"(8)",shift right=2.3em]
\& {[}X^\vee,F(X^\vee \ast_1{[}X^\vee,(X\ot X^\vee)\ast_1 A{]}_1){]}_2
\dar["{[X^\vee,F(X^\vee \ast_1 [X^\vee,c_{X,1}\ast_1 A]_1)]_2}"description]
\rar["h_3"]
\arrow[r,phantom,"(9)",shift right=2.3em]
\& {[}X^\vee,(X\ot X^\vee) \ast_2 FA{]}_2
\dar["{[X^\vee,c_{X,1}\ast_2 FA]_2}"description]
\\
F{[}X^\vee,1\ast_1A{]}_1
\dar["F{[X^\vee,l^1_A]_1}"']
\rar["f_4"]
\arrow[r,phantom,"(10)",shift right=2.3em]
\& {[}X^\vee,X^\vee \ast_2 F{[}X^\vee,1\ast_1 A{]}_1{]}_2
\dar["{[X^\vee,X^\vee\ast_2 F{[X^\vee,l^1_A]_1}]_2}"description]
\rar["g_4"]
\arrow[r,phantom,"(11)",shift right=2.3em]
\& {[}X^\vee,F(X^\vee \ast_1 {[}X^\vee,1\ast_1 A{]}_1){]}_2
\dar["{[X^\vee,F(X^\vee \ast_1 [X^\vee,l^1_A]_1)]_2}"description]
\rar["h_4"]
\arrow[r,phantom,"(12)",shift right=2.3em]
\& {[}X^\vee,1 \ast_2 FA{]}_2
\dar["{[X^\vee,l^2_{FA}]_2}"description]
\\
F{[}X^\vee,A{]}_1
\rar["f_5"]
\& {[}X^\vee,X^\vee\ast_2 F{[}X^\vee,A{]}_1{]}_2
\rar["g_5"]
\& {[}X^\vee,F(X^\vee \ast_1 {[}X^\vee,A{]}_1){]}_2
\rar["h_5"]
\& {[}X^\vee,FA{]}_2
\end{tikzcd}&
\end{flalign*}
It is clear that the squares $(1), (2), \ldots, (12)$ commute. Thus, it remains to show that $u^2_{X^\vee,X\ast_2 FA} \circ \phi_{X,A}=h_1 \circ g_1 \circ f_1 \circ Fu^1_{X^\vee,X\ast_1 A}$. Indeed,
\begin{align*}
h_1 \circ g_1 \circ f_1 \circ Fu^1_{X^\vee,X\ast_1 A}
&=[X^\vee,X^\vee \ast_2 \phi_{X,A}]_2 \circ [X^\vee,\phi_{X^\vee, X\ast_1 A}]_2
\\
&\circ [X^\vee,Fc^1_{X^\vee \ast_1 (X\ast_1 A)}]_2 \circ 
[X^\vee,\phi^{-1}_{X^\vee,{[}X^\vee,X^\vee \ast_1 (X\ast_1 A){]}_1}]_2
\\
&\circ u^2_{X^\vee,F{[X^\vee,X^\vee\ast_1(X\ast_1 A)]_1}} \circ Fu^1_{X^\vee,X\ast_1 A}
\\
&=[X^\vee,X^\vee \ast_2 \phi_{X,A}]_2 \circ [X^\vee,\phi_{X^\vee, X\ast_1 A}]_2
\\
&\circ [X^\vee,Fc^1_{X^\vee \ast_1 (X\ast_1 A)}]_2 \circ [X^\vee,\phi^{-1}_{X^\vee,{[}X^\vee,X^\vee \ast_1 (X\ast_1 A){]}_1}]_2
\\
&\circ [X^\vee,X^\vee\ast_2 Fu^1_{X^\vee,X\ast_1 A}]_2\circ u^2_{X^\vee,F(X\ast_1 A)}
\\
&=[X^\vee,X^\vee \ast_2 \phi_{X,A}]_2 \circ [X^\vee,\phi_{X^\vee, X\ast_1 A}]_2
\\
&\circ [X^\vee,Fc^1_{X^\vee \ast_1 (X\ast_1 A)}]_2 \circ [X^\vee,F(X^\vee\ast_1 u^1_{X^\vee,X\ast_1 A})]_2
\\
&\circ [X^\vee,\phi^{-1}_{X^\vee,X\ast_1 A}]_2 \circ u^2_{X^\vee,F(X\ast_1 A)}
\\
&=[X^\vee,X^\vee \ast_2 \phi_{X,A}]_2 \circ [X^\vee,\phi_{X^\vee, X\ast_1 A}]_2
\\
&\circ [X^\vee,\phi^{-1}_{X^\vee,X\ast_1 A}]_2 \circ u^2_{X^\vee,F(X\ast_1 A)}
\\
&=[X^\vee,X^\vee \ast_2 \phi_{X,A}]_2 \circ u^2_{X^\vee,F(X\ast_1 A)}
\\
&=u^2_{X^\vee,X\ast_2 FA} \circ \phi_{X,A}.
\end{align*}
We conclude that diagram~\eqref{diag:middle} commutes and, as a result, square~\eqref{cd:nt} commutes. If $X\in \scT^\rmc$, then by~\cite[Lemma 4.6]{Stevenson13}, $\mathrm{ev}^i_{X,-}$ is an isomorphism. Consequently, the restriction of $\xi_{-,B}\colon - \ast_1 GB \to G(-\ast_2 B)$ to the compact objects of $\scT$ is a natural isomorphism. Since the triangulated functors $-\ast_1 GB$ and $G(-\ast_2 B)$ are coproduct-preserving, it follows by~\Cref{lem:restrictions-to-compacts} that $\xi_{-,B}$ is a natural isomorphism. It is easy to verify that the conditions of~\Cref{defn:action-pres} are satisfied. We conclude that $G$ is action-preserving.

The proof that if $F$ preserves products, then $F$ is $\shom$-preserving is similar and left to the interested reader.
\end{proof}

Let $\scK$ be a $\scT$-module. A subcategory $\scL\subseteq \scK$ is called a \emph{localizing} \emph{submodule} if~$\scL$ is a localizing subcategory such that $X\ast A\in \scL,\, \forall X\in \scT,\, \forall A\in \scL$. The collection of localizing submodules of $\scK$ is denoted by $\Loca(\scK)$. A subcategory $\scC \subseteq \scK$ is called a \emph{colocalizing} $\shom$-\emph{submodule} if $\scC$ is a colocalizing subcategory such that $[X,A]_\ast\in \scC,\, \forall X\in \scT,\, \forall A\in \scC$. The collection of colocalizing $\shom$-submodules of~$\scK$ is denoted by $\Coloch(\scK)$. Let $A$ be an object of $\scK$. The localizing (resp.~colocalizing) submodule of $\scK$ generated (resp.~cogenerated) by $A$, i.e., the smallest localizing (resp.~colocalizing) submodule of $\scK$ that contains $A$, is denoted by $\loca(A)$ (resp.~$\coloch(A)$). Specializing to the case $\scK=\scT$ and $\ast=\ot$, we obtain the notions of localizing tensor-ideal and colocalizing left $\shom$-ideal.

The subcategory $\Ann_\scT(\scK)\coloneqq\set{X\in \scT}{X\ast -=0}=\bigcap_{A\in \scK}\Ker(-\ast A)\subseteq \scT$ is a localizing tensor-ideal of $\scT$, which is called the \emph{annihilator} of $\scK$ in $\scT$. If $\Ann_\scT(\scK)=0$ (for instance, when $\scK=\scT$ and $\ast=\ot$) then $\scK$ is called a \emph{conservative} $\scT$-module.

\begin{lem}\label{lem:images-of-ideals}
Let $\scK_1,\scK_2$ be two $\scT$-modules and let $\scA\subseteq \Ob(\scK_1)$ and $\scB\subseteq \Ob(\scK_2)$.
\begin{enumerate}[\rm(a)]
\item
If $F\colon \scK_1\to \scK_2$ is a coproduct and action-preserving triangulated functor, then $F(\loca(\scA))\subseteq \loca(F\scA)$.
\item
If $F\colon \scK_1^{\mathrm{op}} \to \scK_2$ is a triangulated functor that sends coproducts to products and $F(X\ast A)\cong [X,FA]_{\ast},\, \forall X\in \scT,\, \forall A\in \scK_1$, then $F(\loca(\scA))\subseteq \coloch(F\scA)$.
\item
If $G\colon \scK_2 \to \scK_1$ is a product and $\shom$-preserving triangulated functor, then $G(\coloch(\scB))\subseteq \coloch(G\scB)$.
\end{enumerate}
\end{lem}

\begin{proof}
We will prove $(\rma)$. The subcategory $\scX=\set{A\in \scK_1}{FA\in \loca(F\scA)}$ is a localizing submodule of $\scK_1$ that contains $\scA$. Therefore, $\scX$ contains $\loca(\scA)$, proving the statement. The proofs of $(\rmb)$ and $(\rmc)$ are similar.
\end{proof}

\subsection{Balmer--Favi support}
Let $\Spc$ be the Balmer spectrum; see~\cite{Balmer05}. Then a point $\mfp\in \Spc$ is called \emph{visible} (\emph{weakly visible} in~\cite{BarthelHeardSanders23}) if there exist Thomason subsets $V,W$ of $\Spc$ such that $\{\mfp\}=V\cap (\Spc\setminus W)$~\cite{BalmerFavi11,Stevenson13}. Particularly, if $\Spc$ is noetherian, then every point of $\Spc$ is visible. The subsets $V$ and $W$ correspond to thick tensor-ideals $\scT^\rmc_V,\, \scT^\rmc_W$ of compact objects. Let $\scT_V=\loct(\scT^\rmc_V)$ and $\scT_W=\loct(\scT^\rmc_W)$. (It should be noted that the localizing subcategories generated by $\scT^\rmc_V$ and $\scT^\rmc_W$ are already tensor-ideals.) Since the ideals $\scT_V$ and $\scT_W$ are compactly generated, they are smashing ideals~\cite{Miller92}. Therefore, they have associated left and right idempotents $e_V,f_V$ and $e_W,f_W$, respectively. Let $g_\mfp=e_V\ot f_W$. Then the objects $\set{g_\mfp}{\mfp\in \Spc \text{ and } \mfp \text{ is visible}}$ are pairwise-orthogonal tensor-idempotents. Let $X$ be an object of $\scT$ and assume that all points of $\Spc$ are visible. Then the \emph{Balmer--Favi} support of $X$ is $\Supp(X)=\set{\mfp\in \Spc}{g_\mfp \ot X\neq 0}$.
\begin{lem}[{\cite[Proposition 7.17]{BalmerFavi11}}]
The map
\[
\Supp(-)\colon \Ob(\scT) \to \scP(\Spc),\ X\mapsto \Supp(X)
\]
satisfies the following properties:
\begin{enumerate}[\rm(a)]
\item
$\Supp(0)=\varnothing \quad \& \quad \Supp(1)=\Spc$.
\item
$\Supp(\coprod_{i\in I}X_i)=\bigcup_{i\in I}\Supp(X_i)$.
\item
$\Supp(\gS X)=\Supp(X)$.
\item
$\Supp(Y)\subseteq \Supp(X)\cup \Supp(Z)$, for any triangle $X\to Y\to Z$.
\item
$\Supp(X\ot Y)\subseteq \Supp(X)\cap \Supp(Y)$.
\item
$\Supp(x\ot y)=\Supp(x)\cap \Supp(y),\ \forall x,y\in \scT^\rmc$.
\end{enumerate}
\end{lem}

Let $V\subseteq \Spc$ be a Thomason subset with complement $U$. The quotient $\scT/\scT_V$ is denoted by $\scT(U)$. Since $\scT_V$ is smashing, $\scT(U)$ is a big tt-category and moreover, $\mathrm{Spc}(\scT(U)^\rmc)\cong U$; see~\cite[Proposition 1.11]{BalmerFavi07}.

\subsection{Smashing support}
Following~\cite{BalchinStevenson23}, we briefly recall some facts concerning the smashing spectrum of a big tt-category $\scT$. We denote by $\Sm(\scT)$ the lattice of smashing tensor-ideals of $\scT$. Then, under the hypothesis that $\Sm(\scT)$ is a spatial frame, there is a space $\Spcs$ associated with $\Sm(\scT)$ via Stone duality, called the \emph{smashing spectrum} of~$\scT$ that consists of the \emph{meet-prime smashing ideals} of $\scT$, i.e., those smashing ideals~$P$ such that $\scS_1\cap \scS_2\subseteq P \Rightarrow \scS_1 \subseteq P$ or $\scS_2\subseteq P,\, \forall \scS_1,\scS_2\in \Sm(\scT)$. The open subsets of $\Spcs$ stand in bijection with the smashing tensor-ideals and are of the form $U_\scS=\set{P\in \Spcs}{\scS\nsubseteq P}$, where $\scS\in \Sm(\scT)$. The closed subsets of $\Spcs$ are of the form $V_\scS=\set{P\in \Spcs}{\scS\subseteq P}$. A point $P\in \Spcs$ is called \emph{locally closed} if $\{P\}=U_\scS\cap V_\scR$, for some smashing ideals $\scS,\scR$. If $P$ is locally closed, then the ideal $\scR$ can be replaced by $P$ in the sense that $\{P\}=U_\scS\cap V_P$. Each smashing ideal $\scS$ corresponds to a left idempotent $e_\scS$ and a right idempotent $f_\scS$, which are the images of the tensor-unit of $\scT$ under the associated acyclization and localization functors, respectively; see~\cite{BalmerFavi11}. If~$P$ is locally closed and $\{P\}=U_\scS\cap V_P$, then the \emph{Rickard idempotent} associated with~$P$ is $\Gamma_P=e_\scS\ot f_P$. If every point of $\Spcs$ is locally closed, then $\Spcs$ is called $T_D$. Let $X$ be an object of $\scT$. The \emph{big smashing support} of $X$ is the subset $\supps(X)=\set{P\in \Spcs}{X\notin P}$. If $\Spcs$ is $T_D$, then the \emph{small smashing support} of $X$ is $\Supps(X)=\set{P\in \Spcs}{\Gamma_P \ot X\neq 0}$. It holds that $\Supps(X)\subseteq \supps(X)$, with the two being equal if $X\in \scT^\rmc$. The analogous properties stated in the following lemma hold for the small smashing support as well.

\begin{lem}[{\cite[Lemma 3.2.8]{BalchinStevenson23}}]
\label{lem:Supps-properties}%
The map
\[
\supps(-)\colon \Ob(\scT)\to \scP(\Spcs),\ X\mapsto \supps(X)
\]
satisfies the following properties:
\begin{enumerate}[\rm(a)]
\item
\label{property:Supp0-Supp1}%
$\supps(0)=\varnothing \quad \& \quad \supps(1)=\Spcs$.
\item
\label{property:Supp-coprod}%
$\supps(\coprod_{i\in I}X_i)=\bigcup_{i\in I}\supps(X_i)$.
\item
\label{property:Supp-suspension}%
$\supps(\gS X)=\supps(X)$.
\item
\label{property:Supp-triangles}%
$\supps(Y)\subseteq \supps(X)\cup \supps(Z)$, for any triangle $X\to Y\to Z$.
\item
\label{property:Supp-weak-tpf}%
$\supps(X\ot Y)\subseteq \supps(X)\cap \supps(Y)$.
\item
\label{property:Supp-compacts}%
$\supps(x\ot y)=\supps(x)\cap \supps(y),\ \forall x,y\in \scT^\rmc$.
\end{enumerate}
\end{lem}
There is a surjective continuous map $\psi\colon \Spcs \to \mathrm{Spc}(\scT^\rmc)^\vee$, where $\mathrm{Spc}(\scT^\rmc)^\vee$ denotes the Hochster dual of $\Spc$, that sends $P\in \Spcs$ to $P\cap \scT^\rmc$. According to~\cite[Corollary 5.1.5]{BalchinStevenson23}, the map $\psi$ is a homeomorphism if and only if $\scT$ satisfies the Telescope Conjecture (meaning that every smashing ideal is generated by compact objects). For further discussion on the smashing spectrum and related concerns, see also~\cite{Verasdanis23}.

\begin{hyp}\label{hyp:hyp-spatial}%
Throughout the paper, whenever we state results concerning the smashing spectrum $\Spcs$, we will always assume that the frame $\Sm(\scT)$ of smashing ideals of $\scT$ is a spatial frame.
\end{hyp}

\section{Stratification--costratification}\label{sec:costrat}

Fix a big tt-category $\scT$ and a $\scT$-module $\scK$. We always assume that $\scK$ is compactly generated. Let us recall some well-known facts concerning Brown--Comenetz duals of compact objects. Let $x$ be a compact object of $\scT$. Then $H_x\coloneqq\Hom_\bbZ(\Hom_\scT(x,-),\bbQ/\bbZ)\colon \scT^\rmop\to \Ab$ is a cohomological functor that sends coproducts to products. So, by Brown Representability, $H_x$ is representable. The representing object of $H_x$ is denoted by $Ix$ and is called the \emph{Brown--Comenetz dual} of $x$. The functor $\Hom_\scT(-,Ix)\rvert_{\scT^\rmc}$ is an injective object of $\Modc$, the abelian category of additive functors $\tcop\to \Ab$; see~\cite{Neeman98}. Hence, by~\cite[Theorem 1.8]{Krause00} (see also~\cite[Theorem 8.6]{Beligiannis00}) $Ix$ is pure-injective. Choosing a skeleton for the subcategory of compact objects, the product of the associated Brown--Comenetz duals is denoted by $I$. Being a product of pure-injective objects, $I$ is also pure-injective. Using the fact that $\scT$ is compactly generated, one can easily check that $I$ is a cogenerator of $\scT$ in the sense that $\Hom_\scT(X,\gS^nI)=0,\, \forall n\in \bbZ$ implies that $X=0$. It holds that $\scT=\coloc(I)$. This follows from the fact that the Brown--Comenetz duals of the compact objects form a perfect cogenerating set for $\scT$; see~\cite{Krause02}. We use the symbol $I_\scT$ if there is any possibility for confusion. Similarly, $\scK$ has a pure-injective cogenerator $I_\scK$.

One could choose any object $X\in \scT$ (resp.~$A\in \scK$) such that $\scT=\coloc(X)$ (resp.~$\scK=\coloc(A)$) and state everything that follows with respect to $X$ (resp.~$A$) instead of $I_\scT$ (resp.~$I_\scK$). The discussion above simply provides a concrete construction of such cogenerators.

\subsection{Support--cosupport}
Fix a topological space $S$.
\begin{defn}[{See also~\cite[Definition 7.1]{BarthelHeardSanders23}}]
A \emph{support data for $\scT$ with values in} $S$ is a map $\rms\colon \Ob(\scT)\to \scP(S)$ that satisfies the following properties:
\begin{enumerate}[\rm(a)]
\item
$\rms(0)=\varnothing \quad \& \quad \rms(1)=S$.
\item
$\rms(\coprod X_i)=\bigcup \rms(X_i)$.
\item
$\rms(\gS X)=\rms(X)$.
\item
$\rms(Y)\subseteq \rms(X)\cup \rms(Z)$, for any triangle $X\to Y\to Z$.
\item
$\rms(X\ot Y)\subseteq \rms(X)\cap \rms(Y)$.
\end{enumerate}
\end{defn}

\begin{defn}
A \emph{cosupport data for $\scT$ with values in} $S$ is a map $\rmc\colon \Ob(\scT)\to \scP(S)$ that satisfies the following properties:
\begin{enumerate}[\rm(a)]
\item
$\rmc(0)=\varnothing \quad \& \quad \rmc(I)=S$.
\item
$\rmc(\prod X_i)=\bigcup \rmc(X_i)$.
\item
$\rmc(\gS X)=\rmc(X)$.
\item
$\rmc(Y)\subseteq \rmc(X)\cup \rmc(Z)$, for any triangle $X\to Y\to Z$.
\item
$\rmc([X,Y])\subseteq \rmc([X,I])\cap \rmc(Y)$.
\end{enumerate}
\end{defn}

\begin{rem}
Let $\rmc\colon \Ob(\scT)\to \scP(S)$ be a cosupport data. If $[-,I]$ preserves triangles (see~\Cref{hyp:tc3}) it is a straightforward verification, using the properties of $[-,I]$, that setting $\rms(X)=\rmc([X,I])$ gives rise to a support data $\rms\colon \Ob(\scT)\to \scP(S)$.
\end{rem}

\begin{defn}
Let $\rmc\colon \Ob(\scT)\to \scP(S)$ be a cosupport data. The support data $\rms\colon \Ob(\scT)\to \scP(S)$ defined by $\rms(X)=\rmc([X,I])$ is called \emph{the support data induced by}~$\rmc$. We say that $(\rms,\rmc)$ is a \emph{support--cosupport pair}.
\end{defn}

\begin{lem}\label{lem:Gamma-supp-cosupp}
Let $\Gamma\colon S\to \Ob(\scT)$ be a map such that $\Gamma_s\coloneqq \Gamma(s)\neq 0,\, \forall s\in S$. Then the maps
\begin{align*}
\rms_\Gamma \colon \Ob(\scT) \to \scP(S),\, \rms_\Gamma(X)&=\set{s\in S}{\Gamma_s\ot X\neq 0}\\
\mathmakebox[\widthof{$\rms_\Gamma$}]{\rmc_\Gamma} \colon \Ob(\scT) \to \scP(S),\, \mathmakebox[\widthof{$\rms_\Gamma$}]{\rmc_\Gamma}(X)&=\set{s\in S}{\mathmakebox[\widthof{$\Gamma_s\ot X$}]{[\Gamma_s,X]}\neq 0}
\end{align*}
are a support and cosupport data, respectively. Moreover, $\rms_\Gamma$ is induced by $\rmc_\Gamma$.
\end{lem}

\begin{proof}
That $\sg$ and $\cg$ are a support and cosupport data, respectively, follows from the fact that $\gG_s\ot -$ is a coproduct-preserving triangulated functor and $[\gG_s,-]$ is a product-preserving triangulated functor. Let $X\in \scT$. The claim that $\sg$ is induced by $\cg$ follows from the isomorphism $[\gG_s,[X,I]]\cong [\gG_s\ot X,I]$ and the fact that $I$ is a cogenerator of $\scT$.
\end{proof}

\begin{defn}\label{defn:good-supp}
A support--cosupport pair $(\rms_\Gamma,\rmc_\Gamma)$ is called \emph{good} if it is induced by a map $\Gamma\colon S\to \Ob(\scT)$ such that $\Gamma_s\ot \Gamma_r=0,\, \forall s\neq r$ and $\Gamma_s\ot \Gamma_s\cong \Gamma_s\neq 0,\, \forall s\in S$.
\end{defn}

\begin{rem}
An important feature of a good support--cosupport pair $(\rms_\Gamma,\rmc_\Gamma)$ is that $\rms_\Gamma(\Gamma_s)=\cg([\Gamma_s,I])=\{s\},\, \forall s\in S$.
\end{rem}

\begin{ex}
The following are good support--cosupport pairs on $\scT$:
\begin{enumerate}[\rm(a)]
\item
Assuming that the frame of smashing ideals of $\scT$ is a spatial frame, the big smashing support--cosupport pair $(\supps,\cosupps)$:
\begin{align*}
\mathmakebox[\widthof{$\cosupps(X)$}]{\supps(X)}&=\set{P\in \Spcs}{f_P\ot X\neq 0},
\\
\cosupps(X)&=\set{P\in \Spcs}{\mathmakebox[\widthof{$f_P\ot X\neq 0$}]{[f_P,X]\neq 0}}.
\end{align*}
Since $\Ker(f_P\ot -)=P$, it holds that $P\in \supps(X)$ if and only if $X\notin P$. Using the equality $P^\perp=\im(f_P\ot -)$, one can deduce that $\Ker [f_P,-]=(P^\perp)^\perp$, so $P\in \cosupps(X)$ if and only if $X\notin (P^\perp)^\perp$.
\item
Assuming further that $\Spcs$ is $T_D$, the small smashing support--cosupport pair $(\Supps,\Cosupps)$:
\begin{align*}
\mathmakebox[\widthof{$\Cosupps(X)$}]{\Supps(X)}&=\set{P\in \Spcs}{\gG_P\ot X\neq 0},
\\
\Cosupps(X)&=\set{P\in\Spcs}{\mathmakebox[\widthof{$\gG_P\ot X\neq 0$}]{[\gG_P,X]\neq 0}}.
\end{align*}
\item
Assuming that every point of $\Spc$ is visible, the Balmer--Favi support--cosupport pair $(\Supp,\Cosupp)$:
\begin{align*}
\mathmakebox[\widthof{$\Cosupp(X)$}]{\Supp(X)}&=\set{\mfp \in \Spc}{g_\mfp \ot X\neq 0},
\\
\Cosupp(X)&=\set{\mfp\in \Spc}{\mathmakebox[\widthof{$g_\mfp \ot X\neq 0$}]{[g_\mfp,X]\neq 0}}.
\end{align*}
\item If $R$ is a graded commutative noetherian ring and $\scT$ is $R$-linear, the BIK support--cosupport pair $(\supp_R,\cosupp_R)$:
\begin{align*}
\mathmakebox[\widthof{$\cosupp_R(X)$}]{\supp_R(X)}&=\set{\mfp \in \Spec(R)}{\gG_\mfp 1\ot X\neq 0},
\\
\cosupp_R(X)&=\set{\mfp\in \Spec(R)}{\mathmakebox[\widthof{$\gG_\mfp 1\ot X\neq 0$}]{[\gG_\mfp 1,X]\neq 0}}.
\end{align*}
See~\cite{BensonIyengarKrause11a,BensonIyengarKrause12}.
\end{enumerate}
\end{ex}

\begin{rem}
Suppose that $\Spcs$ is $T_D$ and let $P\in \Spcs$ with associated idempotent $\gG_P=e_\scS\ot f_P$. If $X$ is an object of $\scT$ such that $P\in \Cosupps(X)$, then $0\neq[\gG_P,X]=[e_\scS\ot f_P,X]\cong [e_\scS,[f_P,X]]$. Hence, $[f_P,X]\neq 0$. In other words, $P\in \cosupps(X)$. This shows that $\Cosupps(X) \subseteq \cosupps(X),\ \forall X\in \scT$.
\end{rem}

\begin{lem}
Assuming that $\Spcs$ is $T_D$, the small and big smashing cosupports coincide, i.e., $\Cosupps(X)=\cosupps(X),\ \forall X\in \scT$, if and only if every point of $\Spcs$ is closed, i.e., $\Spcs$ is $T_1$.
\end{lem}

\begin{proof}
Since $\Supps(-)=\Cosupps([-,I])$ and $\supps(-)=\cosupps([-,I])$, if the small and big smashing cosupports coincide, then so do the small and big smashing supports. By~\cite[Lemma 4.28]{Verasdanis23}, it follows that $\Spcs$ is $T_1$. Conversely, if $\Spcs$ is $T_1$ and $P\in \Spcs$, then $V_P=\{P\}$. This implies that $\gG_P=f_P$. So, $[\gG_P,X]=0$ if and only if $[f_P,X]=0$, for all $X\in \scT$. Hence, $\Cosupp=\cosupps$.
\end{proof}

\begin{rem}
Assume that $\Spcs$ is $T_D$ and consider the small smashing support--cosupport. Then $\Cosupps(1)=\set{P\in \Spcs}{[\gG_P,1]\neq 0}$. There are many cases where $\Cosupps(1)\neq \Spcs$. For instance, $\Cosupps(\bbZ_p)=\{(p)\}\neq \mathrm{Spc}^\rms(\rmD(\bbZ_p))\cong \Spec(\bbZ_p)=\{(0),(p)\}$. For more examples and results concerning the cosupport in derived categories of commutative noetherian rings, see~\cite{Thompson18}.
\end{rem}

Let $(\sg,\cg)$ be a support--cosupport pair on $\scT$ induced by a map $\Gamma\colon S\to \Ob(\scT)$ and define the maps
\begin{align*}
\sga \colon \Ob(\scK) \to \scP(S),\, \sga(A)&=\set{s\in S}{\mathmakebox[\widthof{$[\Gamma_s,A]_\ast$}]{\Gamma_s\ast A}\neq 0},\\
\mathmakebox[\widthof{$\sga$}]{\cga} \colon \Ob(\scK) \to \scP(S),\, \mathmakebox[\widthof{$\sga$}]{\cga}(A)&=\set{s\in S}{[\Gamma_s,A]_\ast \neq 0}.
\end{align*}
\begin{lem}\label{lem:induced-sc}
The maps $\sga$ and $\cga$ satisfy the following properties:
\begin{enumerate}[\rm(a)]
\item
$\sga(0)=\varnothing$.
\item
$\sga(\coprod A_i)=\bigcup \sga(A_i)$.
\item
$\sga(\gS A)=\sga(A)$.
\item
$\sga(B)\subseteq \sga(A)\cup \sga(C)$, for any triangle $A\to B\to C$ of $\scK$.
\item
$\sga(X\ast A)\subseteq \sg(X)\cap \sga(A)$.
\medbreak
\item
$\cga(0)=\varnothing$.
\item
$\cga(\prod A_i)=\bigcup \cga(A_i)$.
\item
$\cga(\gS A)=\cga(A)$.
\item
$\cga(B)\subseteq \cga(A)\cup \cga(C)$, for any triangle $A\to B\to C$ of $\scK$.
\item
$\cga([X,A]_\ast)\subseteq \cg([X,I_\scT])\cap \cga(A)$.
\end{enumerate}
\end{lem}

\begin{proof}
The argument is essentially the same as the one given in~\Cref{lem:Gamma-supp-cosupp}. The property $\cga([X,A]_\ast)\subseteq \cg([X,I_\scT])\cap \cga(A)$ follows from~\Cref{lem:hom-associativity}.
\end{proof}

\subsection{The (co)local-to-global principle and (co)minimality}
We define two pairs of inclusion-preserving maps
\[
\begin{tikzcd} \scP(S)  \rar[shift left,"\tausga"] & \lar[shift left,"\sigmasga"] \Loca(\scK) \end{tikzcd}
\quad \& \quad
\begin{tikzcd} \scP(S)  \rar[shift left,"\taucga"] & \lar[shift left,"\sigmacga"] \Coloch(\scK) \end{tikzcd}
\]
by the formulas
\[
\tausga(W)=\set{A\in \scK}{\sga(A)\subseteq W} \quad \& \quad \sigmasga(\scL)=\bigcup_{A\in \scL}\sga(A),
\]
\[
\taucga(W)=\set{A\in \scK}{\cga(A)\subseteq W} \quad \& \quad \sigmacga(\scC)=\bigcup_{A\in \scC}\cga(A).
\]
It is clear from the properties of $\sga$ and $\cga$ that the maps $\tausga,\sigmasga,\taucga,\sigmacga$ are well-defined. Moreover, $\im \sigmasga \subseteq \scP(\sigmasga(\scK))$ and $\im \sigmacga \subseteq \scP(\sigmacga(\scK))$. In fact, $\sigmasga(\scK)=\sigmacga(\scK)=\cga(I_\scK)$. The first equality follows from the adjunction $\gG_s \ast - \dashv [\gG_s,-]_\ast$. The second equality is a special case of~\Cref{lem:cosupport-of-collection} using the fact that $\scK=\coloch(I_\scK)$. If $\scK$ is a conservative $\scT$-module, then $\gG_s \ast -\neq 0,\, \forall s\in S$. Hence, in this case, $\sigmasga(\scK)=S$.
For $\scK=\scT$ and $\ast=\ot$, we obtain the maps $\tausg,\sigmasg,\taucg,\sigmacg$.

\begin{defn}\label{defn:strat--costrat}%
$\phantom{}$
\begin{enumerate}[\rm(a)]
\item
$\scK$ is \emph{stratified} by $\gG$ if $\tausga$ and $\sigmasga$, between $\scP(\cga(I_\scK))$ and $\Loca(\scK)$, are mutually inverse bijections.
\item
$\scK$ is \emph{costratified} by $\gG$ if $\taucga$ and $\sigmacga$, between $\scP(\cga(I_\scK))$ and $\Coloch(\scK)$, are mutually inverse bijections.
\end{enumerate}
\end{defn}

\begin{rem}
Since we will always work with a fixed support--cosupport pair induced by a map $\gG\colon S\to \Ob(\scT)$, we will omit the reference to $\gG$ in~\Cref{defn:strat--costrat} and say ``$\scK$ is stratified'' and ``$\scK$ is costratified'', respectively. We will mention explicit support--cosupport pairs where appropriate.
\end{rem}

\begin{ex}[{\cite{Krause05,Stevenson14b}}]
Let $R$ be a commutative noetherian ring. Then the singularity category $\rmS(R)=\rmK_{\mathrm{ac}}(\Inj R)$, which is the homotopy category of acyclic complexes of injective $R$-modules, is a compactly generated triangulated category and there is an action $ \ast  \colon \rmD(R) \times \rmS(R) \to \rmS(R)$: If $X\in \rmD(R)$ and $A\in \rmS(R)$, then $X\ast A=\wt{X}\ot_R A$, where $\wt{X}$ is a K-flat resolution of $X$. In this instance, $\sigmasga(\rmS(R))$ is $\mathrm{Sing}(R)$ the singular locus of $R$. If $R$ is locally a hypersurface, then $\rmS(R)$ is stratified by the action of $\rmD(R)$: There is a bijective correspondence between localizing subcategories of $\rmS(R)$ and subsets of $\mathrm{Sing}(R)$; see~\cite[Theorem 6.13]{Stevenson14b}.
\end{ex}

\begin{defn}[{For $(\rma)$, see~\cite[Definition 6.1]{Stevenson13}}]
$\phantom{}$
\begin{enumerate}[\rm(a)]
\item
$\scK$ satisfies the \emph{local-to-global principle} if
\[
\loca(A)=\loca\parens{\Gamma_s\ast A}{s\in S},\, \forall A\in \scK.
\]
\item
$\scK$ satisfies \emph{minimality} if, for all $s\in S$, $\loca\parens{\Gamma_s\ast A}{A\in \scK}$ is minimal in $\Loca(\scK)$ in the sense that it does not contain any non-zero proper localizing submodule of $\scK$.
\end{enumerate}
\begin{enumerate}[\rm(a)]
\item[$(\mathrm{i})$]
$\scK$ satisfies the \emph{colocal-to-global principle} if
\[
\coloch(A)=\coloch\parens{[\Gamma_s,A]_\ast}{s\in S},\, \forall A\in \scK.
\]
\item[$(\mathrm{ii})$]
$\scK$ satisfies \emph{cominimality} if, for all $s\in S$, $\coloch([\Gamma_s,I_\scK]_\ast)$ is minimal in $\Coloch(\scK)$ in the sense that it does not contain any non-zero proper colocalizing $\shom$-submodule of $\scK$.
\end{enumerate}
\end{defn}

\begin{rem}\label{rem:comin-expanded}
Let $X\in \scT$. Since $\scK=\coloc(I_\scK)=\coloch(I_\scK)$, by~\Cref{lem:images-of-ideals} for the functor $[X,-]_\ast$, we have $\coloch([X,I_\scK]_\ast)=\coloch\parens{[X,A]_\ast}{A\in \scK}$. In particular, $\coloch([\Gamma_s,I_\scK]_\ast)=\coloch\parens{[\Gamma_s,A]_\ast}{A\in \scK},\ \forall s\in S$.
\end{rem}

\begin{rem}\label{rem:coloc-optimal}
It is clear from the definition of $\sga$ and $\cga$ that
\begin{align*}
\loca\parens{\Gamma_s\ast A}{s\in S}&=\loca\parens{\Gamma_s\ast A}{s\in \sga(A)},\\
\coloch\parens{[\Gamma_s,A]_\ast}{s\in S}&=\coloch\parens{[\Gamma_s,A]_\ast}{s\in \cga(A)}.
\end{align*}
In addition, if $\scK$ satisfies the  local-to-global (resp.~colocal-to-global) principle, then $\sga$ (resp.~$\cga$) detects vanishing, i.e., $\sga(A)=\varnothing \Rightarrow A=0$ and similarly for $\cga$. For the case $\scK=\scT$ and $\ast=\ot$, it holds that codetection implies detection, since $\varnothing=\sg(X)=\cg([X,I])$ implies $[X,I]=0$, so $X=0$.
\end{rem}

For the rest of the section, fix a good support--cosupport pair $(\sg,\cg)$ on $\scT$.

\begin{lem}\label{lem:cosupport-of-collection}
Let $\scA$ be a collection of objects of $\scK$. Then we have the following equalities of subsets of $S$:
\begin{enumerate}[\rm(a)]
\item
$\sigmasga(\loca(\scA))=\bigcup_{A\in \scA}\sga(A)$.
\item
$\sigmacga(\coloch(\scA))=\bigcup_{A\in \scA}\cga(A)$.
\end{enumerate}
\end{lem}

\begin{proof}
We will prove the case of $\sigmacga$. Let $s$ be an element of $S$. Then
\begin{align*}
s\notin \bigcup_{A\in \scA}\cga(A) &\Leftrightarrow \scA \subseteq \Ker [\Gamma_s,-]_\ast\\
&\Leftrightarrow \coloch(\scA)\subseteq \Ker [\Gamma_s,-]_\ast\\
&\Leftrightarrow s\notin \bigcup_{A\in \coloch(\scA)}\cga(A)=\sigmacga(\coloch(\scA)).\qedhere
\end{align*}
\end{proof}

\begin{rem}\label{rem:key-prop}
If $(\sg,\cg)$ is a good support-cosupport pair on $\scT$, then it holds that $\cga([\Gamma_s,A]_\ast)\subseteq \{s\}$. Hence, if $[\Gamma_s,A]_\ast \neq 0$ (i.e., $s\in \cga(A)$) then $\cga([\Gamma_s,A]_\ast)=\{s\}$. In particular, if $s\in \cga(I_\scK)$, then $\cga([\Gamma_s,I_\scK]_\ast)=\{s\}$.
\end{rem}

\begin{lem}\label{lem:tau-inj-sigma-surj}
It holds that $\sigmasga\circ \tausga=\Id$ and $\sigmacga\circ \taucga=\Id$, where both composites are restricted to $\scP(\cga(I_\scK))$. In particular, the respective restrictions of $\tausga$ and $\taucga$ are injective, while $\sigmasga$ and $\sigmacga$ are surjective.
\end{lem}

\begin{proof}
We will prove that $\sigmacga\circ \taucga=\Id$ (restricted to $\scP(\cga(I_\scK)$)). Let $W$ be a subset of $\cga(I_\scK)$. Clearly $(\sigmacga \circ \taucga)(W)\subseteq W$, since $(\sigmacga\circ \taucga)(W)=\bigcup_{\cga(A)\subseteq W}\cga(A)$. Let $s$ be an element of $W$. Then $s\in \cga([\gG_s,I_\scK]_\ast)=\{s\}\subseteq W$. Therefore, $s\in (\sigmacga \circ \taucga)(W)$, completing the proof.
\end{proof}

\begin{thm}\label{thm:costrat-coltg-comin}
Let $(\rms_\Gamma,\rmc_\Gamma)$ be a good support-cosupport pair on $\scT$.
\begin{enumerate}[\rm(a)]
\item
\label{thm:ccc-a}
$\scK$ is stratified with respect to $(\sg,\cg)$ if and only if $\scK$ satisfies the local-to-global principle and minimality.
\item
\label{thm:ccc-b}
$\scK$ is costratified with respect to $(\sg,\cg)$ if and only if $\scK$ satisfies the colocal-to-global principle and cominimality.
\end{enumerate}
\end{thm}

\begin{proof}
We will only prove $(\rmb)$, since $(\rma)$ is proved analogously. Suppose that $\scK$ is costratified. Then $\sigmacga$ is injective. Let $A$ be an object of $\scK$. Then
\begin{align*}
\sigmacga(\coloch\parens{[\gG_s,A]_\ast}{s\in \cga(A)})&=\bigcup_{s\in \cga(A)}\cga([\gG_s,A]_\ast)\\
&=\bigcup_{s\in \cga(A)}\{s\}\\
&=\cga(A)\\
&=\sigmacga(\coloch(A)),
\end{align*}
where the first and last equalities are due to~\Cref{lem:cosupport-of-collection}. Since $\sigmacga$ is injective, it follows that $\coloch\parens{[\gG_s,A]_\ast}{s\in \cga(A)}=\coloch(A)$. Thus, $\scK$ satisfies the colocal-to-global principle. In particular, $\cga$ detects vanishing.

Let $s$ be an element of $\cga(I_\scK)$ and $A$ a non-zero object in $\coloch([\gG_s,I_\scK]_\ast)$. Then $\varnothing \neq \cga(A) \subseteq \cga([\gG_s,I_\scK]_\ast)=\{s\}$. Therefore, $\cga(A)=\cga([\gG_s,I_\scK]_\ast)$. Since $\sigmacga$ is injective, $\coloch(A)=\coloch([\gG_s,I_\scK]_\ast)$. Hence, $\coloch([\gG_s,I_\scK]_\ast)$ is minimal.

Suppose that $\scK$ satisfies the colocal-to-global principle and cominimality. Let $\scC\in \Coloch(\scK)$. Clearly, $\scC\subseteq (\taucga \circ \sigmacga)(\scC)$. Let $A\in (\taucga \circ \sigmacga)(\scC)$, i.e., $\cga(A)\subseteq \sigmacga(\scC)$. Then
\begin{align*}
\coloch(A)&=\coloch\parens{[\gG_s,A]_\ast}{s\in \cga(A)}\\
&\subseteq\coloch\parens{[\gG_s,I_\scK]_\ast}{s\in \cga(A)}\\
&\subseteq \coloch\parens{[\gG_s,I_\scK]_\ast}{s\in \sigmacga(\scC)}\\
&\subseteq \scC.
\end{align*}
The first equality is due to the colocal-to-global principle. The first containment relation follows from~\Cref{rem:comin-expanded}, while the second containment relation is clear. For the third containment, if $s\in \sigmacga(\scC)$, then there exists an object $B\in \scC$ such that $[\gG_s,B]_\ast\neq 0$. Since $[\gG_s,B]_\ast\in \coloch([\gG_s,I_\scK]_\ast)$ and the latter is minimal, it follows that $\coloch([\gG_s,I_\scK]_\ast)=\coloch([\gG_s,B]_\ast)\subseteq \scC$. We infer that $A\in \scC$, proving that $(\taucga \circ \sigmacga)(\scC)=\scC$. So, $\sigmacga$ is injective and thus, $\sigmacga$ is bijective. This shows that $\scK$ is costratified.
\end{proof}

\begin{rem}\label{rem:replacing-good}
\Cref{thm:costrat-coltg-comin}~\eqref{thm:ccc-b} could be stated slightly more generally, replacing a good support--cosupport pair $(\sg,\cg)$, in the sense of~\Cref{defn:good-supp}, with one that satisfies the property stated in~\Cref{rem:key-prop}, i.e., if $A$ is an object of $\scK$ such that $[\gG_s,A]_\ast\neq 0$, then $\cga([\gG_s,A]_\ast)=\{s\}$. Similarly, the analogous property for~\Cref{thm:costrat-coltg-comin}~\eqref{thm:ccc-a} is: if $A$ is an object of $\scK$ such that $\gG_s \ast A\neq 0$, then $\sga(\gG_s \ast A)=\{s\}$. This observation will be useful in~\Cref{sec:derived}, where we consider the support--cosupport for objects of the derived category of a commutative noetherian ring defined by the residue fields.
\end{rem}

\subsection{Local-to-global implies colocal-to-global}
Let $(\sg,\cg)$ be a (not necessarily good) support--cosupport pair on $\scT$. For our next result, we need an additional assumption.
\begin{hyp}\label{hyp:tc3}
We further assume that the relative internal-hom of $\scK$ is a triangulated functor in the first variable, i.e., $[-,A]_\ast\colon \scT^\rmop \to \scK$ preserves triangles, for all $A\in \scK$. This is true, e.g., if $\scK$ satisfies a formulation of May's TC3 axiom~(\cite{May01}) replacing the tensor product of $\scT$ with the action of $\scT$ on $\scK$. The proof of~\cite[Theorem C.1]{Murfet07} goes through verbatim. Our assumption is satisfied by all known examples.
\end{hyp}

\begin{lem}\label{lem:generators-ideals-submodules}
Suppose that $\scT=\loct(G)$. Then the following hold:
\begin{enumerate}[\rm(a)]
\item
$\loca(A)=\loca(G\ast A),\, \forall A\in \scK$.
\item
$\coloch(A)=\coloch([G,A]_\ast),\, \forall A\in \scK$ (under~\Cref{hyp:tc3}).
\end{enumerate}
\end{lem}

\begin{proof}
We will prove $(\rmb)$. The inclusion $\coloch([G,A]_\ast)\subseteq \coloch(A)$ is clear. Since $\scT=\loct(G)$, it holds that $1\in \loct(G)$. By~\Cref{lem:images-of-ideals} for the functor $[-,A]_\ast$, it follows that $A\cong[1,A]_\ast\in \coloch([G,A]_\ast)$. Therefore, $\coloch(A)\subseteq \coloch([G,A]_\ast)$. The proof of $(\rma)$ is analogous.
\end{proof}

\begin{rem}\label{rem:gen-t-gen-k}
An easy generalization of~\Cref{lem:generators-ideals-submodules} is the following: If $\scT=\loct(\scG)$, for a collection of objects $\scG$, then $\forall A\in \scK\colon \loca(A)=\loca\parens{G\ast A}{G\in \scG}$ and $\coloch(A)=\coloch\parens{[G,A]_\ast}{G\in \scG}$.
\end{rem}

\begin{prop}[{See also~\cite[Proposition 6.8]{Stevenson13}}]\label{prop:ltg-implies-ltg-coltg}
If $\scT$ satisfies the local-to-global principle, then $\scK$ satisfies the local-to-global principle and (under~\Cref{hyp:tc3}) the colocal-to-global principle.
\end{prop}

\begin{proof}
Since $\scT$ satisfies the local-to-global principle, we have $\scT=\loct\parens{\gG_s}{s\in S}$. Hence, by~\Cref{rem:gen-t-gen-k}, $\loca(A)=\loca\parens{\Gamma_s\ast A}{s\in S}$ and $\coloch(A)=\coloch\parens{[\gG_s,A]_\ast}{s\in S}$, for all $A\in \scK$. This proves the statement.
\end{proof}

\begin{cor}\label{cor:ltg-implies-coltg}
Under~\Cref{hyp:tc3} for the case $\scK=\scT$, if $\scT$ satisfies the local-to-global principle, then $\scT$ satisfies the colocal-to-global principle.
\end{cor}

\begin{ex}
Let $R$ be a graded commutative noetherian ring such that $\scT$ is $R$-linear and consider the BIK support--cosupport $(\supp_R,\cosupp_R)$, which takes values in $\supp_R(1)\subseteq\Spec(R)$ --- this may not be an equality. As explained in~\cite[Corollary 7.11]{BarthelHeardSanders23}, if $\scT$ is stratified in the sense of BIK, then $\supp_R(1)$ is homeomorphic to $\Spc$ and the BIK support is identified with the Balmer--Favi support under this homeomorphism. It then follows that $\scT$ is stratified by the Balmer--Favi support. Now since the tensor-idempotents $\gG_\mfp 1$ (defining the BIK support) and the tensor-idempotents $g_\mfp$ (defining the Balmer--Favi support) have the same support (which is $\{\mfp\}$) it follows that $\loct(\gG_\mfp 1)=\loct(g_\mfp)$. Applying~\Cref{lem:images-of-ideals} for the functor $[-,I]$ (taking into account~\Cref{hyp:tc3}) it follows that $\coloch([\gG_\mfp 1,I])=\coloch([g_\mfp,I])$. By~\Cref{cor:ltg-implies-coltg}, $\scT$ satisfies the colocal-to-global principle with respect to the Balmer--Favi support. Taking into account~\Cref{thm:costrat-coltg-comin}, we conclude that if $\scT$ is BIK-stratified, then: $\scT$ is Balmer--Favi-costratified if and only if $\scT$ is BIK-costratified if and only if $\coloch([\gG_\mfp 1,I])$ is minimal, for all $\mfp\in \supp_R(1)$. If $\scT=\Modu(kG)$ is the stable module category of the group algebra of a finite group $G$, then $\scT$ is BIK-costratified by the canonical action of $H^\ast(G,k)$; see~\cite[Theorem 11.13]{BensonIyengarKrause12}. We infer that $\Modu(kG)$ is Balmer--Favi-costratified.
\end{ex}

\section{Prime submodules}\label{sec:primes}

In this section we introduce the classes of prime localizing submodules and $\shom$-prime colocalizing submodules of a given $\scT$-module $\scK$. The class of prime localizing submodules generalizes the class of objectwise-prime localizing tensor-ideals~\cite{BalchinStevenson23,Verasdanis23} in the context of relative tensor-triangular geometry, while the class of $\shom$-prime colocalizing submodules specializes to the class of $\shom$-prime colocalizing left $\shom$-ideals if $\scK=\scT$.

\subsection{Prime localizing and colocalizing submodules}
As before, $(\rms_\Gamma,\rmc_\Gamma)$ will be a good support-cosupport pair on $\scT$ with values in a space $S$. Given $\scL \in \Loca(\scK)$ and $\scC\in \Coloch(\scK)$, we define two subcategories of $\scT$ as follows:
\begin{align*}
\scL^{\ot \rmL}&=\set{X\in \scT}{\mathmakebox[\widthof{$[X,\scK]_\ast\subseteq \scC$}]{X\ast \scK \subseteq \scL}},
\\
\scC^{\ot\rmC}&=\set{X\in \scT}{[X,\scK]_\ast\subseteq \scC},
\end{align*}
where $X\ast \scK\coloneqq \loca\parens{X\ast A}{A\in \scK}$ and $[X,\scK]_\ast \coloneqq \coloch\parens{[X,A]_\ast}{A\in \scK}$, with the latter also being equal to $\coloch([X,I_\scK]_\ast)$. Evidently, if $\scL_1\subseteq \scL_2$, then $\scL_1^{\ot \rmL}\subseteq \scL_2^{\ot \rmL}$ and if $\scC_1\subseteq \scC_2$, then $\scC_1^{\ot \rmC}\subseteq \scC_2^{\ot \rmC}$.

\begin{rem}\label{rem:c-cones}
Clearly, $\scL^{\ot \rmL}$ is a localizing tensor-ideal of $\scT$ and $\scC^{\ot \rmC}$ is closed under coproducts, suspensions and the tensor product. Under~\Cref{hyp:tc3}, $\scC^{\ot \rmC}$ is also closed under cones and so, $\scC^{\ot \rmC}$ is a localizing tensor-ideal of $\scT$.
\end{rem}

\begin{rem}\label{rem:lot-self-action}
If $\scK=\scT$ and $\ast=\ot$, then $\scL^{\ot\rmL}=\scL$. The inclusion $\scL^{\ot\rmL}\subseteq \scL$ follows from the equality $X\ot \scT=\loct(X)$, while the inclusion $\scL \subseteq \scL^{\ot \rmL}$ holds because $\scL$ is a tensor-ideal.
\end{rem}

\begin{defn}
$\phantom{}$
\begin{enumerate}[\rm(a)]
\item
A proper localizing submodule $\scL\subseteq \scK$ is called \emph{prime} if $X\ast A\in \scL$ implies $X\in \scL^{\ot \rmL}$ or $A\in \scL$.
\item
A proper colocalizing $\shom$-submodule $\scC\subseteq \scK$ is called \emph{$\shom$-prime} if $[X,A]_\ast\in \scC$ implies $X\in \scC^{\ot \rmC}$ or $A\in \scC$.
\end{enumerate}
\end{defn}

\begin{rem}
If $\scK=\scT$ and $\ast=\ot$, then the notion of prime localizing submodule recovers the notion of objectwise-prime localizing tensor-ideal; see~\Cref{rem:lot-self-action}. The notion of $\shom$-prime colocalizing $\shom$-submodule provides the notion of $\shom$-prime colocalizing left $\shom$-ideal.
\end{rem}

\begin{lem}
Let $\scL$ be a prime localizing submodule of $\scK$ and let $\scC$ be a $\shom$-prime colocalizing submodule of $\scK$. Then $\scL^{\ot \rmL}$ and $\scC^{\ot \rmC}$ are objectwise-prime, in the sense that if $X\ot Y\in \scL^{\ot \rmL}$, then $X\in \scL^{\ot \rmL}$ or $Y\in \scL^{\ot \rmL}$ and similarly for $\scC^{\ot \rmC}$.
\end{lem}

\begin{proof}
We will prove that $\scC^{\ot \rmC}$ is objectwise-prime. The proof for $\scL^{\ot \rmL}$ is analogous. Let $X,Y$ be objects of $\scT$ such that $X\ot Y\in \scC^{\ot \rmC}$. Then $[X\ot Y,A]_\ast \in \scC,\ \forall A\in \scK$. By~\Cref{lem:hom-associativity}, $[X,[Y,A]_\ast]_\ast \cong [X\ot Y,A]_\ast$. Since $\scC$ is $\shom$-prime, $X\in \scC^{\ot \rmC}$ or $[Y,A]_\ast\in \scC$. If $X\notin \scC^{\ot \rmC}$, then $[Y,A]_\ast\in \scC,\ \forall A\in \scK$, i.e., $Y\in \scC^{\ot \rmC}$. This proves that $\scC^{\ot \rmC}$ is objectwise-prime.
\end{proof}

The main result of this section, i.e.,~\Cref{thm:costrat-hom-primes}, is a consequence of the following series of lemmas.

\begin{lem}\label{lem:ker-otl-otc}
The following statements hold:
\begin{enumerate}[\rm(a)]
\item
$\Ker(\gG_s\ot -) \subseteq \Ker(\gG_s\ast -)^{\ot\rmL}=\Ker([\gG_s,-]_\ast)^{\ot \rmC},\, \forall s\in S$.
\item If $\scK$ is conservative, then $\Ker(\gG_s\ot -) = \Ker(\gG_s\ast -)^{\ot\rmL},\, \forall s\in S$.
\item
If $\Ker(\gG_s\ot -) = \Ker(\gG_s\ast -)^{\ot\rmL},\, \forall s\in S$ and $\sg$ detects vanishing, then $\scK$ is conservative.
\end{enumerate}
\end{lem}

\begin{proof}
Let $X$ be an object of $\scT$. Then we have $X\in \Ker(\gG_s\ast -)^{\ot\rmL}$ if and only if $\gG_s\ast (X\ast A)\cong (\gG_s\ot X)\ast A=0,\ \forall A\in \scK$, which is equivalent to $(\gG_s \ot X) \ast - =0$. Similarly, using the isomorphism $[\gG_s\ot X,-]_\ast \cong [\gG_s,[X,-]_\ast]_\ast$, one deduces that $X\in \Ker([\gG_s,-]_\ast)^{\ot \rmC}$ if and only if $[\gG_s\ot X,-]_\ast =0$. Since $(\gG_s \ot X) \ast - \dashv [\gG_s \ot X,-]_\ast$, these two functors are either both the zero functor on $\scK$ or none of them is the zero functor. Therefore, $\Ker(\gG_s\ast -)^{\ot\rmL}=\Ker([\gG_s,-]_\ast)^{\ot \rmC}$. Since $\Ker(\gG_s\ast -)^{\ot\rmL}=\set{X\in \scT}{(\gG_s\ot X)\ast - =0}$, it immediately follows that $\Ker(\gG_s\ot -) \subseteq \Ker(\gG_s\ast -)^{\ot\rmL}$. This proves $(\rma)$.

If $\scK$ is conservative and $(\gG_s\ot X)\ast -=0$, then $\gG_s \ot X=0$. Hence, $\Ker(\gG_s\ot -)=\Ker(\gG_s\ast -)^{\ot\rmL}$. This proves $(\rmb)$.

Let $X\in \scT$ such that $X\ast -=0$. Then $(\Gamma_s \ot X)\ast - \cong X \ast (\Gamma_s \ast -)=0,\, \forall s\in S$. Therefore, $X\in \Ker(\gG_s\ast -)^{\ot\rmL}$. This implies that $X\in \Ker(\Gamma_s \ot -),\, \forall s\in S$, i.e., $\Gamma_s \ot X=0,\, \forall s\in S$. Equivalently, $\sg(X)=\varnothing$. Since $\sg$ detects vanishing, it follows that $X=0$. This proves $(\rmc)$.
\end{proof}

\begin{lem}\label{lem:obj-contained}
$\phantom{}$
\begin{enumerate}[\rm(a)]
\item
Let $\scL$ be a prime localizing submodule of $\scK$. There is at most one $s\in \cga(I_\scK)$ such that $\scL\subseteq \Ker(\Gamma_s \ast -)$.
\item
Let $\scC$ be a $\shom$-prime colocalizing submodule of $\scK$. There is at most one $s\in \cga(I_\scK)$ such that $\scC\subseteq \Ker[\Gamma_s, -]_\ast$.
\end{enumerate}
\end{lem}

\begin{proof}
$\phantom{}$
\begin{enumerate}[\rm(a)]
\item
Similar to $(\rmb)$.
\item
Let $s\in \cga(I_\scK)$ and suppose that $\scC\subseteq \Ker[\gG_s,-]_\ast$. Let $r\in S$ such that $r\neq s$ and let $A\in \scK$. Then $[\Gamma_s,[\Gamma_r,A]_\ast]_\ast=0\in \scC$. Since $\scC$ is $\shom$-prime, $\Gamma_s \in \scC^{\ot \rmC} \subseteq \Ker[\Gamma_s, -]_\ast^{\ot \rmC} = \Ker (\Gamma_s \ast -)^{\ot \rmL}$ or $[\Gamma_r,A]_\ast\in \scC\subseteq \Ker[\Gamma_s, -]_\ast$. For the equality $\Ker[\Gamma_s, -]_\ast^{\ot \rmC} = \Ker (\Gamma_s \ast -)^{\ot \rmL}$, see~\Cref{lem:ker-otl-otc}. The former of the two does not hold since $\Gamma_s\in \Ker(\Gamma_s \ast -)^{\ot \rmL}$ if and only if $\Gamma_s \ast -=0$, but $s\in \cga(I_\scK)$ which means that $\Gamma_s \ast -\neq 0$. It follows that $\scC$ contains all objects $[\Gamma_r,A]_\ast$, for $r\neq s$ and $A\in \scK$. So, if $\scC\subseteq \Ker [\Gamma_r,-]_\ast$, for $r\neq s$ and $r\in \cga(I_\scK)$, then $[\Gamma_r,A]_\ast\in \Ker[\gG_r,-]_\ast,\, \forall A\in \scK$. It follows that $[\Gamma_r,-]_\ast=0$; thus, $\gG_r\ast -=0$, which is false since $r\in \cga(I_\scK)$.\qedhere
\end{enumerate}
\end{proof}

\begin{lem}\label{lem:coltg-kernels}
If $\scK$ satisfies the colocal-to-global principle, then it holds $\forall s \in S \colon$ $\Ker [\gG_s,-]_\ast=\coloch\parens{[\gG_r,I_\scK]_\ast}{r\neq s}$.
\end{lem}

\begin{proof}
Let $r,s\in S$ such that $r\neq s$. Then $[\Gamma_r,I_\scK]_\ast\in \Ker[\Gamma_s,-]_\ast$. Therefore, $\coloch\parens{[\gG_r,I_\scK]_\ast}{r\neq s}\subseteq \Ker [\gG_s,-]_\ast$. Let $A\in \Ker[\Gamma_s,-]_\ast$. Then $s\notin \cga(A)$. Since $\scK$ satisfies the colocal-to-global principle,
\begin{align*}
\coloch(A)&=\coloch\parens{[\Gamma_r,A]_\ast}{r\in \cga(A)}\\
&=\coloch\parens{[\Gamma_r,A]_\ast}{r\neq s}\\
&\subseteq \coloch\parens{[\gG_r,I_\scK]_\ast}{r\neq s}.
\end{align*}
See~\Cref{rem:coloc-optimal} for the second equality and~\Cref{rem:comin-expanded} for the containment relation. Hence, $A\in \coloch\parens{[\gG_r,I_\scK]_\ast}{r\neq s}$, completing the proof.
\end{proof}

\begin{lem}\label{lem:intersection-ker}%
Let $\scC\in \Coloch(\scK)$. Then $\taucga(\sigmacga(\scC))=\bigcap_{\substack{\scC\subseteq \Ker [\gG_s,-]_\ast \\ s\in \cga(I_\scK)}} \Ker[\gG_s,-]_\ast$. If $\scK$ is costratified, then $\scC=\bigcap_{\substack{\scC\subseteq \Ker [\gG_s,-]_\ast \\ s\in \cga(I_\scK)}} \Ker[\gG_s,-]_\ast$.
\end{lem}

\begin{proof}
Let $A\in \scK$. Then $A\notin \bigcap_{\scC\subseteq \Ker [\gG_s,-]_\ast} \Ker[\gG_s,-]_\ast$ if and only if there exists $s\in S$ such that $\scC\subseteq \Ker [\gG_s,-]_\ast$ and $[\gG_s,A]_\ast \neq 0$. Equivalently, $s\notin \sigmacga(\scC)$ and $s\in \cga(A)$. In other words, $\cga(A) \nsubseteq \sigmacga(\scC)$. Since $\taucga(\sigmacga(\scC))$ consists precisely of those $A\in \scK$ such that $\cga(A)\subseteq \sigmacga(\scC)$, it follows that $\taucga(\sigmacga(\scC))=\bigcap_{\scC\subseteq \Ker [\gG_s,-]_\ast} \Ker[\gG_s,-]_\ast$. Finally, if $\scK$ is costratified, then $\scC=\taucga(\sigmacga(\scC))$, which proves the statement (the indexing set of the intersection involved in the claimed equalities can be considered to consist of points $s\in \cga(I_\scK)$ since if $s\notin \cga(I_\scK)$, then $[\Gamma_s,-]_\ast=0$ and so $\Ker[\Gamma_s,-]_\ast=\scK$ so the intersection is not affected).
\end{proof}

\begin{thm}\label{thm:costrat-hom-primes}
Let $\scK$ be a costratified $\scT$-module. Then there is a bijective correspondence between $\shom$-prime colocalizing submodules of $\scK$ and points of $\cga(I_\scK)$. A point $s\in \cga(I_\scK)$ is associated with $\Ker[\Gamma_s,-]_\ast=\coloch\parens{[\gG_r,I_\scK]_\ast}{r\neq s}$.
\end{thm}

\begin{proof}
Let $\scC\in \Coloch(\scK)$ be $\shom$-prime. Then $\scC=\bigcap_{\substack{\scC \subseteq \Ker[\Gamma_s,-]_\ast \\ s\in \cga(I_\scK)}}\Ker[\Gamma_s,-]_\ast$, by~\Cref{lem:intersection-ker}. It follows by~\Cref{lem:obj-contained} that $\scC$ must be contained in $\Ker[\Gamma_s,-]_\ast$, for a unique $s\in \cga(I_\scK)$. Conclusion: $\scC=\Ker[\Gamma_s,-]_\ast$, for a unique $s\in \cga(I_\scK)$. The equality $\Ker[\Gamma_s,-]_\ast=\coloch\parens{[\gG_r,I_\scK]_\ast}{r\neq s}$ was proved in~\Cref{lem:coltg-kernels}.
\end{proof}

The following observation, which is of independent interest and will not play a role in the sequel, showcases a conceptual similarity between the theory of actions of tensor-triangulated categories and the theory of associated primes of modules over rings. To see this, recall the following result: If $R$ is a ring and $M$ is a non-zero $R$-module such that for every non-zero submodule $N\subseteq M$, it holds that $\Ann_R(M)=\Ann_R(N)$, then $\Ann_R(M)$ is a prime ideal of $R$.

\begin{prop}
Let $\scL$ be a non-zero localizing submodule of $\scK$ such that for every non-zero localizing submodule $\scL'$ of $\scK$ with $\scL'\subseteq \scL$, it holds that $\Ann_\scT(\scL)=\Ann_\scT(\scL')$. Then $\Ann_\scT(\scL)$ is an objectwise-prime localizing ideal of $\scT$.
\end{prop}

\begin{proof}
Let $X,Y\in \scT$ such that $X\ot Y\in \Ann_\scT(\scL)$. This means that $(X\ot Y) \ast \scL=0$. Suppose that $X\notin \Ann_\scT(\scL)$, i.e., $X\ast \scL\neq 0$. Then $\Ann_\scT(X\ast \scL)=\Ann_\scT(\scL)$. Since $Y\ast (X\ast \scL)=(X\ot Y) \ast \scL=0$, it follows that $Y\in \Ann_\scT(X\ast \scL)$, so $Y\in \Ann_\scT(\scL)$.
\end{proof}

\subsection{The action and internal-hom formulas}

\begin{defn}
$\phantom{}$
\begin{enumerate}[\rm(a)]
\item
$\scK$ satisfies the \emph{Action Formula} (AF) if
\begin{displaymath}
\sga(X\ast A)=\sg(X)\cap \sga(A),\, \forall X\in \scT,\, \forall A\in \scK.
\end{displaymath}
\item
$\scK$ satisfies the \emph{Internal-Hom Formula} (IHF) if
\begin{displaymath}
\cga([X,A]_\ast)=\cg([X,I_\scT])\cap \cga(A),\, \forall X \in \scT,\, \forall A\in \scK.
\end{displaymath}
(Recall that $\cg([X,I_\scT])=\sg(X)$.)
\end{enumerate}
\end{defn}

\begin{lem}\label{lem:obj-tpf}%
$\phantom{}$
\begin{enumerate}[\rm(a)]
\item
If $\scK$ satisfies the Action Formula, then $\Ker(\Gamma_s\ast -)$ is a prime localizing submodule, $\forall s\in S$. If $\scK$ is a conservative $\scT$-module, then the converse holds.
\item
If $\scK$ satisfies the Internal-Hom Formula, then $\Ker[\Gamma_s,-]_\ast$ is a $\shom$-prime colocalizing $\shom$-submodule, $\forall s\in S$. If $\scK$ is a conservative $\scT$-module, then the converse holds.
\end{enumerate}
\end{lem}

\begin{proof}
$\phantom{}$
\begin{enumerate}[\rm(a)]
\item
Similar to $(\rmb)$.
\item
The Internal-Hom Formula can be restated as follows: if $[\gG_s\ot X,A]_\ast=0$ then $\gG_s\ot X =0$ or $[\gG_s,A]_\ast=0$ --- the converse holds by the definition of cosupport. So, if $[X,A]_\ast \in \Ker[\gG_s,-]_\ast$, then $X\in \Ker(\gG_s\ot -)\subseteq \Ker[\gG_s,-]_\ast^{\ot \rmC}$ or $A\in \Ker[\gG_s,-]_\ast$; for the first alternative, see~\Cref{lem:ker-otl-otc}. This means that $\Ker[\gG_s,-]_\ast$ is $\shom$-prime. Now if $\Ker[\gG_s,-]_\ast$ is $\shom$-prime and $\scK$ is a conservative $\scT$-module, then $\Ker(\gG_s\ot -) = \Ker[\gG_s,-]_\ast^{\ot \rmC}$. Therefore, if $[\gG_s\ot X,A]_\ast=0$, then $\gG_s\ot X=0$ or $[\gG_s,A]_\ast=0$, which is precisely the statement of the Internal-Hom Formula.\qedhere
\end{enumerate}
\end{proof}

\begin{prop}
$\phantom{}$
\begin{enumerate}[\rm(a)]
\item
If $\scT$ satisfies minimality, then $\scK$ satisfies the Action Formula and (under~\Cref{hyp:tc3}) the Internal-Hom Formula.
\item
If $\scK$ is a conservative $\scT$-module and $\scK$ satisfies cominimality, then $\scK$ satisfies the Internal-Hom Formula.
\item
If $\scT$ satisfies the Internal-Hom Formula, then $\scT$ satisfies the Action Formula.
\end{enumerate}
\end{prop}

\begin{proof}
Let $s\in S,\, X\in \scT,\, A\in \scK$.
\begin{enumerate}[\rm(a)]
\item
If $s\in \sg(X)\cap \sga(A)$, then $\Gamma_s\ot X\neq 0$ and $\Gamma_s \ast A\neq 0$. Since $\loct(\Gamma_s)$ is minimal, it follows that $\Gamma_s\in \loct(\Gamma_s \ot X)$. Hence, $\Gamma_s \ast A \in \loca((\Gamma_s \ot X)\ast A)$. Since $\Gamma_s \ast A\neq 0$, it holds that $\Gamma_s \ast (X \ast A) \cong (\Gamma_s \ot X)\ast A \neq 0$. In other words, $s \in \sga(X\ast A)$. Conclusion: $\scK$ satisfies AF.

Now suppose that $s \in \cg([X,I_\scT])\cap \cga(A)$. Then $\Gamma_s \ot X\neq 0$ and $[\Gamma_s,A]_\ast\neq 0$ (recall that $\cg([X,I_\scT])=\sg(X)$). Since $\loct(\Gamma_s)$ is minimal, it follows that $\Gamma_s\in \loct(\Gamma_s \ot X)$. Hence, $[\Gamma_s,A]_\ast \in \coloch([\Gamma_s\ot X,A]_\ast)$. It follows that $[\Gamma_s\ot X,A]_\ast\neq 0$, i.e., $s \in \cga([X,A]_\ast)$. Conclusion: $\scK$ satisfies IHF.
\item
Suppose that $s \in \cg([X,I_\scT])\cap \cga(A)$. Then $\Gamma_s \ot X\neq 0$ and $[\Gamma_s,A]_\ast\neq 0$. Aiming for contradiction, assume that $[\Gamma_s\ot X,A]_\ast=0$. Then $A\in \Ker[\Gamma_s \ot X,-]_\ast$. Since $0\neq [\Gamma_s,A]_\ast\in \coloch([\Gamma_s,I_\scK]_\ast)$, it follows by cominimality of $\scK$ that $\coloch([\Gamma_s,I_\scK]_\ast)=\coloch([\Gamma_s,A]_\ast)\subseteq \coloch(A)\subseteq \Ker[\Gamma_s \ot X,-]_\ast$. Thus, $[\Gamma_s\ot X,I_\scK]_\ast=[\Gamma_s\ot X,[\Gamma_s,I_\scK]_\ast]_\ast =0$. So, $[\gG_s\ot X,-]_\ast = 0$. By the $(\gG_s\ot X) \ast - \dashv [\gG_s \ot X,-]_\ast$ adjunction, it follows that $(\gG_s\ot X) \ast - =0$, i.e., $\Gamma_s\ot X\in \Ann_\scT(\scK)$. Since $\scK$ is a conservative $\scT$-module, $\Gamma_s\ot X=0$, which is a contradiction. Conclusion: $\scK$ satisfies IHF.
\item
Let $X,Y\in \scT$. Then $\sg(X\ot Y)=\cg([X\ot Y,I_\scT])=\cg([X,[Y,I_\scT]])= \cg([X,I_\scT])\cap \cg([Y,I_\scT])=\sg(X)\cap \sg(Y)$. Conclusion: IHF implies AF.\qedhere
\end{enumerate}
\end{proof}

\begin{rem}
If $\scK=\scT$, then the statement of the Action Formula is: $\sg(X\ot Y)=\sg(X)\cap \sg(Y),\ \forall X,Y\in \scT$. This is known as the Tensor Product Formula (which does not hold in general); see~\cite{BalmerFavi11,BarthelHeardSanders23}. See also~\cite{Balmer20} for a support theory that does satisfy the Tensor Product Formula. On the other hand, the Internal-Hom Formula states: $\cg([X,Y])=\sg(X)\cap \cg(Y),\ \forall X,Y\in \scT$. For the BIK support, this is equivalent to stratification of $\scT$; see~\cite[Theorem 9.5]{BensonIyengarKrause12}.
\end{rem}

\section{Smashing submodules}\label{sec:smash}

Let $\scK$ be a $\scT$-module. Recall our assumption that $\scK$ is compactly generated. A smashing submodule of $\scK$ is a smashing subcategory $\scM\subseteq \scK$ that is also a submodule. Specifically, the quotient functor $j_\scM\colon \scK\to \scK/\scM$ is a coproduct-preserving and essentially surjective triangulated functor that has a right adjoint $k_\scM\colon \scK/\scM \to \scK$ (which is necessarily fully faithful) that preserves coproducts --- and products since it is a right adjoint. By Brown representability, $k_\scM$ has a right adjoint $\ell_\scM\colon \scK\to \scK/\scM$ (which is necessarily essentially surjective) that preserves products. By the relations $j_\scM k_\scM\cong \Id \cong \ell_\scM k_\scM$, it follows that $j_\scM$ and $\ell_\scM$ take the same values on the image of $k_\scM$, which is $\scM^\perp$. The set of smashing submodules of~$\scK$ is denoted by $\Sma(\scK)$.

Next we describe the action of $\scT$ on $\scK/\scM$ induced by the action of $\scT$ on $\scK$. The category $\scT\times \scK/\scM$ is a triangulated category that is the quotient of $\scT\times \scK$ over $0\times \scM$, with the quotient functor $\scT\times \scK\to \scT\times \scK/\scM$ being $\Id_\scT\times j_\scM$. Since $0\times \scM$ is contained in the kernel of $j_\scM \circ \ast$, it follows that $j_\scM \circ \ast$ factors through $\scT\times \scK/\scM$ via a functor $\ast\colon \scT\times \scK/\scM \to \scK/\scM$. It is straightforward to check that this functor is an action of~$\scT$ on $\scK/\scM$. If $X\in \scT$ and $A=j_\scM(B)\in \scK/\scM$, then $X\ast A=j_\scM(X\ast B)$. The functor $j_\scM\colon \scK\to \scK/\scM$ is action-preserving. We denote by $[-,-]_\ast\colon \scT^\mathrm{op}\times \scK/\scM\to \scK/\scM$ the relative internal-hom of $\scK/\scM$. By~\Cref{lem:right-adj-action}, $k_\scM$ is action and $\shom$-preserving and $\ell_\scM$ is $\shom$-preserving. Moreover, since $I_\scK$ (the product of the Brown--Comenetz duals of the compact objects of $\scK$) is a pure-injective cogenerator of $\scK$ and $\ell_\scM$ is an essentially surjective right adjoint, it follows that $\ell_\scM(I_\scK)$ is a pure-injective cogenerator of $\scK/\scM$. In particular, $\scK/\scM=\coloc(\ell_\scM(I_\scK))$.

Now we describe the colocalizing $\shom$-submodules of $\scK/\scM$. The functor $k_\scM$ gives a bijective correspondence between the colocalizing subcategories of $\scK/\scM$ and the colocalizing subcategories of $\scK$ contained in $\scM^\perp$. Since $k_\scM$ is $\shom$-preserving, this bijection restricts to colocalizing $\shom$-submodules, i.e., the maps
\begin{equation}\label{eq:coloc-quotient}
\begin{tikzcd}
\Coloch(\scK/\scM) \rar["k_\scM",shift left] & \set{\scC\in \Coloch(\scK)}{\scC\subseteq \scM^\perp} \lar["k_\scM^{-1}",shift left]
\end{tikzcd}
\end{equation}
are mutually inverse inclusion-preserving bijections. An observation that will be useful in the sequel is that $k_\scM \coloch(j_\scM(A))=\coloch(k_\scM j_\scM (A)),\, \forall A\in \scK$.

Let $(\rms_\Gamma,\rmc_\Gamma)$ be a good support--cosupport pair on $\scT$. We denote the induced support--cosupport on $\scK/\scM$ by $(\rms^\scM_\Gamma,\rmc^\scM_\Gamma)$. Specifically,
\begin{align*}
\rms^\scM_\Gamma(j_\scM(A))&=\set{s\in S}{\mathmakebox[\widthof{$[\Gamma_s,j_\scM(A)]_\ast\neq 0$}]{j_\scM(\gG_s\ast A)\neq 0}},
\\
\rmc^\scM_\Gamma(j_\scM(A))&=\set{s\in S}{[\Gamma_s,j_\scM(A)]_\ast\neq 0}.
\end{align*}
Then $\scK/\scM$ satisfies the colocal-to-global principle if
\[
\coloch(j_\scM(A))=\coloch\parens{[\gG_s,j_\scM(A)]_\ast}{s\in S},\ \forall A\in \scK
\]
and $\scK/\scM$ satisfies cominimality if $\coloch([\gG_s,\ell_\scM(I_\scK)]_\ast)$ is a minimal colocalizing $\shom$-submodule of $\scK/\scM$, for all $s\in S$. Finally, let $S_\scM=\set{s\in S}{[\gG_s,I_\scK]_\ast \in \scM^\perp}$.

\begin{prop}\label{prop:coltg-local}
Let $\scM\in \Sma(\scK)$. The following are equivalent:
\begin{enumerate}[\rm(a)]
\item
$\scK/\scM$ satisfies the colocal-to-global principle.
\item
$\coloch(B)=\coloch\parens{[\Gamma_s,B]_\ast}{s\in S},\, \forall B\in \scM^\perp$.
\end{enumerate}
As a result, if $\scK$ satisfies the colocal-to-global principle, then $\scK/\scM$ satisfies the colocal-to-global principle.
\end{prop}

\begin{proof}
Let $A$ be an object of $\scK$ and set
\begin{align*}
\scC_1&=\coloch(j_\scM(A)),
\\
\scC_2&=\coloch\parens{[\gG_s,j_\scM(A)]_\ast}{s\in S},
\\
\scD_1&=\coloch(k_\scM j_\scM (A)),
\\
\scD_2&=\coloch\parens{[\gG_s,k_\scM j_\scM (A)]_\ast}{s\in S}.
\end{align*}
Under the bijection~\eqref{eq:coloc-quotient}, $\scC_1$ corresponds to $\scD_1$, while $\scC_2$ corresponds to $\scD_2$ (recall that $k_\scM$ is $\shom$-preserving). So, if $\scK/\scM$ satisfies the colocal-to-global principle, then $\scC_1=\scC_2$. Hence, $\scD_1=\scD_2$. Since $\im k_\scM j_\scM=\scM^\perp$, $(\rmb)$ follows. On the other hand, if $(\rmb)$ holds, then $\scD_1=\scD_2$. As a result, $\scC_1=\scC_2$, i.e., $\scK/\scM$ satisfies the colocal-to-global principle. This proves $(\rma)$.

If $\scK$ satisfies the colocal-to-global principle, then
\[
\coloch(A)=\coloch\parens{[\Gamma_s,A]_\ast}{s\in S},\, \forall A\in \scK,
\]
so the equality certainly holds for $A\in \scM^\perp$. Therefore, $\scK/\scM$ satisfies the colocal-to-global principle by the equivalence $(\rma)\Leftrightarrow (\rmb)$.
\end{proof}

\begin{prop}\label{prop:comin-local}
Suppose that $s\in S_\scM$. Then $\coloch([\Gamma_s,I_\scK]_\ast)$ is a minimal colocalizing $\shom$-submodule of $\scK$ if and only if $\coloch([\Gamma_s,\ell_\scM(I_\scK)]_\ast)$ is a minimal colocalizing $\shom$-submodule of $\scK/\scM$.
\end{prop}

\begin{proof}
Since $s\in S_\scM$, it holds that $[\gG_s,I_\scK]_\ast \in \scM^\perp$. Therefore, $[\gG_s,\ell_\scM(I_\scK)]_\ast\cong \ell_\scM[\gG_s,I_\scK]_\ast\cong j_\scM[\gG_s,I_\scK]_\ast$. So, under the bijection~\eqref{eq:coloc-quotient}, $\coloch([\gG_s,\ell_\scM(I_\scK)]_\ast)$ corresponds to $\coloch(k_\scM j_\scM [\gG_s,I_\scK]_\ast)=\coloch([\gG_s,I_\scK]_\ast)$, with the last equality again because $[\gG_s,I_\scK]_\ast\in \scM^\perp$. Consequently, $\coloch([\Gamma_s,I_\scK]_\ast)$ is minimal if and only if $\coloch([\Gamma_s,\ell_\scM(I_\scK)]_\ast)$ is minimal.
\end{proof}

Combining~\Cref{prop:coltg-local},~\Cref{prop:comin-local} and~\Cref{thm:costrat-coltg-comin}, we obtain the following result.

\begin{thm}\label{thm:comin-local}
Let $\{\scM_s\}_{s\in S}$ be a collection of smashing submodules of $\scK$ such that $s\in S_{\scM_s}$, for all $s\in S$. Then:
\begin{enumerate}[\rm(a)]
\item
$\scK$ satisfies cominimality if and only if $\scK/\scM_s$ satisfies cominimality, for all $s \in S$.
\item
Suppose that $\scK$ satisfies the colocal-to-global principle. Then $\scK$ is costratified if and only if $\scK/\scM_s$ is costratified, for all $s\in S$.
\end{enumerate}
\end{thm}

\begin{proof}
Since $s\in S_{\scM_s}$, for all $s\in S$, by~\Cref{prop:comin-local} $\coloch([\Gamma_s,I_\scK]_\ast)$ is a minimal colocalizing $\shom$-submodule of $\scK$ if and only if $\coloch([\Gamma_s,\ell_{\scM_s}(I_\scK)]_\ast)$ is a minimal colocalizing $\shom$-submodule of $\scK/\scM_s$, for all $s\in S$. In other words, $\scK$ satisfies cominimality if and only if $\scK/\scM_s$ satisfies cominimality, for all $s\in S$. This proves $(\rma)$. If $\scK$ satisfies the colocal-to-global principle, then by~\Cref{prop:coltg-local}, it follows that $\scK/\scM_s$ satisfies the colocal-to-global principle, for all $s\in S$. Statement $(\rmb)$ now follows from $(\rma)$ and~\Cref{thm:costrat-coltg-comin}.
\end{proof}

We will apply~\Cref{thm:comin-local} to the case $\scK=\scT$, $\ast=\ot$, $S=\Spcs$ and $(\sg,\cg)=(\Supps,\Cosupps)$ (under~\Cref{hyp:hyp-spatial} and provided that $\Spcs$ is $T_D$). In this case, if $P\in \Spcs$, then $S_P=\set{Q\in \Spcs}{[\gG_Q,I]\in P^\perp}$. Since $\gG_P=e_\scS\ot f_P$, for some $\scS\in \Sm(\scT)$, and $P^\perp=\im[f_P,-]$, it follows that $[\gG_P,I]=[e_\scS\ot f_P,I]\cong [f_P,[e_\scS,I]]\in P^\perp$. In other words, $P\in S_P$. This leads to the following result:

\begin{cor}\label{cor:comin-local-smash}
Suppose that $\Spcs$ is $T_D$. Then:
\begin{enumerate}[\rm(a)]
\item
$\scT$ satisfies cominimality if and only if $\scT/P$ satisfies cominimality, for all $P\in \Spcs$.
\item
Suppose that $\scT$ satisfies the colocal-to-global principle. Then $\scT$ is costratified if and only if $\scT/P$ is costratified, for all $P\in \Spcs$.
\end{enumerate}
\end{cor}

\begin{proof}
The result is a direct consequence of~\Cref{thm:comin-local}, taking into account the preceding discussion.
\end{proof}

\begin{cor}\label{cor:comin-local-smash-covers}
Suppose that $\Spcs$ is $T_D$ and that $\Spcs=\bigcup_{j\in J}V_{\scS_j}$ is a cover of $\Spcs$ by closed subsets. If $\scT/\scS_j$ satisfies cominimality, for all $j\in J$, then $\scT$ satisfies cominimality. If, moreover, $\scT$ satisfies the colocal-to-global principle, then $\scT$ is costratified.
\end{cor}

\begin{proof}
Let $P\in \Spcs$. Then $P\in V_{\scS_j}$, for some $j\in J$. This means that $\scS_j\subseteq P$. Let $j_{\scS_j}\colon \scT\to \scT/\scS_j$ be the quotient functor. Then $j_{\scS_j}(P)$ is a smashing ideal of $\scT/\scS_j$ such that $(\scT/\scS_j)/j_{\scS_j}(P)\simeq \scT/P$. Since $\scT/\scS_j$ satisfies cominimality, it follows by~\Cref{cor:comin-local-smash} that $\scT/P$ satisfies cominimality. Since this is true for all $P\in \Spcs$, again by~\Cref{cor:comin-local-smash}, we conclude that $\scT$ satisfies cominimality. The ``moreover'' part follows by~\Cref{thm:costrat-coltg-comin}.
\end{proof}

Essentially via the same arguments (left to the reader) one obtains the analogous results for the Balmer spectrum and the Balmer--Favi support. What one needs to note for~\Cref{cor:comin-covers-compacts} is that, compared to $\Spcs$ where the smashing ideals stand in bijection with open subsets of $\Spcs$ (thus closed covers of $\Spcs$ are necessary) the thick ideals of $\scT^\rmc$ --- and by extension the compactly generated smashing ideals of $\scT$ --- stand in bijection with Thomason subsets of $\Spc$; hence, a cover by complements of Thomason subsets is what is needed. 

\begin{cor}\label{cor:comin-local-bf}
Suppose that every point of $\Spc$ is visible. Then:
\begin{enumerate}[\rm(a)]
\item
$\scT$ satisfies cominimality if and only if $\scT/\loct(\mfp)$ satisfies cominimality, for all $\mfp\in \Spc$.
\item
Suppose that $\scT$ satisfies the colocal-to-global principle. Then $\scT$ is costratified if and only if $\scT/\loct(\mfp)$ is costratified, for all $\mfp\in \Spc$.
\end{enumerate}
\end{cor}

\begin{cor}\label{cor:comin-covers-compacts}
Suppose that every point of $\Spc$ is visible and that $\Spc=\bigcup_{j\in J}U_j$ is a cover of $\Spc$ by complements of Thomason subsets. Let $V_j$ be the complement of $U_j$. If $\scT(U_j)$ satisfies cominimality, for all $j\in J$, then $\scT$ satisfies cominimality. If, moreover, $\scT$ satisfies the colocal-to-global principle, then $\scT$ is costratified.
\end{cor}

In view of future applications involving singularity categories of schemes, we need a version of~\Cref{cor:comin-covers-compacts} for the more general case where $\scT$ acts on $\scK$. Let $\scS$ be a compactly generated localizing tensor-ideal of $\scT$ and set $\scM=\scS\ast \scK$. Then $\scM$ is a compactly generated localizing submodule of $\scK$; see~\cite[Section 4]{Stevenson13}. The action of $\scT$ on $\scK$ induces, as already discussed previously, an action of $\scT$ on $\scK/\scM$. Because of the way $\scM$ is defined, it follows that there is an induced action of $\scT/\scS$ on $\scK/\scM$ and a colocalizing subcategory of $\scK/\scM$ is a $\shom$ $\scT$-submodule if and only if it is a $\shom$ $\scT/\scS$-submodule.

Assuming that every point of $\Spc$ is visible, let $V$ be a Thomason subset of $\Spc$ and let $U=\Spc\setminus V$ and consider the localizing tensor-ideal $\scT_V$ generated by those compact objects of $\scT$ whose support is contained in $V$. By definition, $\scT_V$ is compactly generated and hence smashing, so there are associated left and right (respectively) idempotents $e_V$ and $f_V$ such that $\scT_V=\loct(e_V)=\Ker(f_V\ot -)=\im(e_V\ot -)$. We denote by $\scT(U)$ the category $\scT/\scT_V$. It holds that $\mathrm{Spc}(\scT(U)^\rmc)\cong U$ and we will treat this homeomorphism as an identification. Let $\scK_V=\scT_V\ast \scK$ and let $\scK(U)=\scK/\scK_V$. By the previous paragraph, $\scK_V$ is a compactly generated localizing submodule of $\scK$ and there is an induced action of $\scT(U)$ on $\scK(U)$ such that a colocalizing subcategory of $\scK(U)$ is a $\shom$ $\scT$-submodule if and only if it is a $\shom$ $\scT(U)$-submodule. Further, $\scK_V=\im(e_V\ast -)=\Ker(f_V\ast -)$ and $\scK_V^\perp=\im[e_V,-]_\ast=\Ker[f_V,-]_\ast$. By this last observation, it follows that $S_{\scK_V}\coloneqq \set{\mfp \in \Spc}{[g_\mfp,I_\scK]_\ast\in \scK_V^\perp}=U$.

The following result is the analogue of~\cite[Theorem 8.11]{Stevenson13} for colocalizing $\shom$-submodules.
\begin{thm}\label{thm:costrat-actions-covers}%
Suppose that every point of $\Spc$ is visible and that $\Spc=\bigcup_{j\in J}U_j$ is a cover of $\Spc$ by complements of Thomason subsets. Let $V_j$ be the complement of $U_j$. If $\scK(U_j)$ (as a $\scT(U_j)$-module) satisfies cominimality, for all $j\in J$, then $\scK$ satisfies cominimality. If, moreover, $\scK$ satisfies the colocal-to-global principle, then $\scK$ is costratified.
\end{thm}

\begin{proof}
If $\mfp \in \Spc$, then there exists $j_\mfp\in J$ such that $\mfp\in U_{j_\mfp}$. Fix such a $j_\mfp\in J$, for each $\mfp\in \Spc$. Then we have a collection $\{\scK_{V_{j_\mfp}}\}_{\mfp \in \Spc}$ of smashing submodules of $\scK$ such that $\mfp \in S_{\scK_{V_{j_\mfp}}}$ since the latter is equal to $U_{j_\mfp}$. The result now follows by an immediate application of~\Cref{thm:comin-local}.
\end{proof}

\section{Derived categories of noetherian rings and schemes}\label{sec:derived}
Throughout, $R$ will denote a commutative noetherian ring. In the article~\cite{Neeman11}, Neeman proved that there is a bijective correspondence between colocalizing subcategories of $\rmD(R)$ and subsets of $\Spec(R)$. In this section, we give a more streamlined proof of Neeman's theorem by using the general machinery we developed; specifically~\Cref{thm:costrat-coltg-comin} and~\Cref{cor:ltg-implies-coltg}. As a direct consequence, we obtain a complete description of the $\RHom$-prime colocalizing subcategories of $\rmD(R)$ in terms of the residue fields. Further, using~\Cref{cor:comin-covers-compacts}, we prove that the derived category of quasi-coherent sheaves over a noetherian separated scheme is costratified.

\begin{rem}\label{rem:dx-tc3}
Let $X$ be a quasi-compact separated scheme. By~\cite[Proposition C.13]{Murfet07}, $\rmD(X)$ the derived category of quasi-coherent sheaves over $X$ satisfies~\Cref{hyp:tc3} (with $\rmD(R)$ being the special case $X=\Spec(R)$). In particular, this allows us to apply~\Cref{cor:ltg-implies-coltg} later.
\end{rem}

\subsection{Noetherian rings}
We will use the cosupport taking values in $\Spec(R)$ defined by the residue fields $k(\mfp)$. More specifically, if $X\in \rmD(R)$, then $\Cosupph(X)=\set{\mfp\in \Spec(R)}{\!\RHom_R(k(\mfp),X)\neq 0}$. We use the notation $\Cosupph$ to avoid conflict with the Balmer--Favi cosupport. Note that since $\rmD(R)$ is generated by its tensor-unit, every colocalizing subcategory of $\rmD(R)$ is a left $\RHom$-ideal. We denote by $I_R$ the cogenerator of $\rmD(R)$ that is the product of the Brown-Comenetz duals of the compact objects.

\begin{lem}\label{lem:residue-summand}
Let $\mfp\in \Spec(R)$. Then $\RHom_R(k(\mfp),X)\cong \bigoplus_{i\in \bbZ}\gS^i k(\mfp)^{(J_i)}\cong \prod_{i\in \bbZ}\gS^i k(\mfp)^{(J_i)}$, for some sets $J_i$, for all $X\in \rmD(R)$. The same holds for the complex $\RHom_R(X,k(\mfp))$.
\end{lem}

\begin{proof}
Let $E$ be a K-injective resolution of $X$. Then $\RHom_R(k(\mfp),X)$ is the $\Hom$-complex $\Hom_R(k(\mfp),E)$. This is a complex of $k(\mfp)$-vector spaces, therefore it must be quasi-isomorphic to its cohomology complex with zero differential (which also has $k(\mfp)$-vector spaces as terms; thus coproducts of copies of $k(\mfp)$). For $\RHom_R(X,k(\mfp))$, pick a K-projective resolution of $X$ instead of a K-injective resolution and argue in an identical manner.
\end{proof}

\begin{lem}\label{lem:key-cominimality}
Let $X$ be an object of $\rmD(R)$ such that $\RHom_R(k(\mfp),X)\neq 0$. Then $\coloc(k(\mfp))=\coloc(\RHom_R(k(\mfp),X))\subseteq \coloc(X)$.
\end{lem}

\begin{proof}
It holds that $\RHom_R(k(\mfp),X)\cong \prod_{i \in \bbZ}\gS^i k(\mfp)^{(J_i)}$. Since $k(\mfp)^{(J_i)} \hookrightarrow k(\mfp)^{J_i}$ is a map of $k(\mfp)$-vector spaces, it must split. So, $k(\mfp)^{(J_i)}$ is a summand of $k(\mfp)^{J_i}$. This implies that $k(\mfp)^{(J_i)}\in \coloc(k(\mfp))$ and consequently, $\prod_{i \in \bbZ}\gS^i k(\mfp)^{(J_i)}\in \coloc(k(\mfp))$. Thus, $\coloc(\RHom_R(k(\mfp),X))\subseteq \coloc(k(\mfp))$. By the fact that $k(\mfp)$ is a summand of $\RHom_R(k(\mfp),X)$, it follows that $\coloc(k(\mfp))\subseteq \coloc(\RHom_R(k(\mfp),X))$. Since $\rmD(R)$ is generated by its tensor-unit, every colocalizing subcategory of $\rmD(R)$ is a left $\RHom$-ideal. Hence, $\RHom_R(k(\mfp),X)\in \coloc(X)$. This completes the proof.
\end{proof}

\begin{prop}\label{prop:cosuppkp}
The category $\rmD(R)$ satisfies the colocal-to-global principle (in particular, $\Cosupph$ detects vanishing) and, for each $\mfp\in \Spec(R)$, it holds that $\Cosupph(k(\mfp))=\{\mfp\}$.
\end{prop}

\begin{proof}
Since $\rmD(R)$ satisfies the local-to-global principle~\cite{Neeman92}, by~\Cref{cor:ltg-implies-coltg}, $\rmD(R)$ satisfies the colocal-to-global principle and, by~\Cref{rem:coloc-optimal}, $\Cosupph$ detects vanishing. Hence, $\Cosupph(k(\mfp))\neq \varnothing$. Let $\mfq\in \Spec(R)$ such that $\mfp\neq \mfq$. By~\Cref{lem:residue-summand}, $\RHom_R(k(\mfp),k(\mfq))$ is quasi-isomorphic to a complex whose terms are of the form $k(\mfp)^{(I)}\cong k(\mfq)^{(J)}$ and these are both $k(\mfp)$ and $k(\mfq)$-vector spaces. Since $\mfp\neq \mfq$, this can only happen if the indexing sets $I$ and $J$ are empty. Hence, $\RHom_R(k(\mfp),k(\mfq))=0$. Consequently, $\Cosupph(k(\mfp))=\{\mfp\}$.
\end{proof}

\begin{thm}[{\cite{Neeman11}}]\label{thm:affine-costrat}
Let $R$ be a commutative noetherian ring. Then $\rmD(R)$ is costratified.
\end{thm}

\begin{proof}
Let $\mfp\in \Spec(R)$ and let $X$ be a non-zero object in $\coloc(k(\mfp))$. Then $\coloc(X)\subseteq \coloc(k(\mfp))$. By~\Cref{prop:cosuppkp}, $\rmD(R)$ satisfies the colocal-to-global principle, $\Cosupph$ detects vanishing and $\Cosupph(k(\mfp))=\{\mfp\}$. By~\Cref{lem:cosupport-of-collection}, it follows that $\Cosupph(X)=\{\mfp\}$, i.e., $\RHom_R(k(\mfp),X)\neq 0$. As a result,~\Cref{lem:key-cominimality} implies that $\coloc(X)= \coloc(k(\mfp))$. So, $\coloc(k(\mfp))$ is a minimal colocalizing subcategory. Moreover,~\Cref{lem:key-cominimality} implies that $\coloc(k(\mfp))=\coloc(\RHom_R(k(\mfp),I_R))$ and so, $\coloc(\RHom_R(k(\mfp),I_R))$ is minimal. In conclusion, $\rmD(R)$ satisfies both the colocal-to-global principle and cominimality; so, \Cref{thm:costrat-coltg-comin} implies that $\rmD(R)$ is costratified; see also~\Cref{rem:replacing-good}. 
\end{proof}

\begin{thm}\label{thm:affine-hom-primes}
The $\RHom$-prime colocalizing subcategories of $\rmD(R)$ correspond to points of $\Spec(R)$. The correspondence is given by associating $\mfp\in \Spec(R)$ with $\Ker\RHom_R(k(\mfp),-)=\coloc\parens{k(\mfq)}{\mfq\neq \mfp}$.
\end{thm}

\begin{proof}
Since $\rmD(R)$ is costratified, and clearly a conservative $\rmD(R)$-module,~\Cref{thm:costrat-hom-primes} implies that the $\RHom$-prime colocalizing subcategories of $\rmD(R)$ are precisely of the form $\Ker\RHom_R(k(\mfp),-)=\coloc\parens{\!\RHom_R(k(\mfq),I_R)}{\mfq\neq \mfp}$ and the claimed equality is due to~\Cref{lem:key-cominimality}.
\end{proof}

\begin{rem}
One could also choose to work with the Balmer--Favi support (or the smashing support since $\rmD(R)$ satisfies the Telescope Conjecture~\cite{Neeman92}; see also~\cite[Lemma 7.2]{Verasdanis23} and~\cite[Section 6]{BalchinStevenson23}). There is a homeomorphism between $\mathrm{Spc}(\rmD^\mathrm{perf}(R))$ and $\Spec(R)$~\cite{Neeman92}. Using this homeomorphism, we can express the Balmer--Favi support--cosupport via $\Spec(R)$. For each $\mfp\in \Spec(R)$, the Balmer--Favi idempotent associated with $\mfp$ is $g_\mfp=K_\infty(\mfp)\ot R_\mfp$, where $K_\infty(\mfp)$ is the stable Koszul complex and $R_\mfp$ is the localization of $R$ at $\mfp$. The objects $g_\mfp$ are orthogonal tensor-idempotents, so they define a support--cosupport pair: Let $X\in \rmD(R)$. Then $\Supp(X)=\set{\mfp\in \Spec(R)}{g_\mfp \ot X\neq 0}$ and $\Cosupp(X)=\set{\mfp\in \Spec(R)}{\RHom_R(g_\mfp,X)\neq 0}$. It holds that $\loc(g_\mfp)=\loc(k(\mfp))$~\cite[Lemma 3.22]{Stevenson18b}. Therefore, $\coloc(\RHom_R(g_\mfp,X))=\coloc(\RHom_R(k(\mfp),X))=\coloc(k(\mfp))$, with the last equality by~\Cref{lem:key-cominimality} (provided that $\RHom_R(k(\mfp),X)\neq 0$). Since $\rmD(R)$ is stratified by the Balmer--Favi support~\cite[Theorem 5.8]{BarthelHeardSanders23}, in particular it satisfies the local-to-global principle, $\rmD(R)$ must also satisfy the colocal-to-global principle; see~\Cref{cor:ltg-implies-coltg}. The equality $\coloc(\RHom_R(g_\mfp,X))=\coloc(k(\mfp))$ shows that $\rmD(R)$ satisfies cominimality with respect to the Balmer--Favi support--cosupport. Therefore, by~\Cref{thm:costrat-coltg-comin}, $\rmD(R)$ is costratified with respect to the Balmer--Favi support--cosupport.
\end{rem}

\begin{ex}[{\cite{Stevenson14a}}]
We include an example of a category that is not costratified. Let $R$ be an absolutely flat ring that is not semi-artinian. Then there exists a superdecomposable injective $R$-module $E$. Let $\mfp\in \Spec(R)$. Then $\RHom_R(k(\mfp),E)=\Hom_R(k(\mfp),E)$. If there was a non-zero map $k(\mfp)\to E$, then (as $k(\mfp)$ is simple and injective since $R$ is absolutely flat) $k(\mfp)$ would have to be a summand of~$E$, which leads to a contradiction. This shows that $\RHom_R(k(\mfp),E)=0$, for all $\mfp \in \Spec(R)$, i.e., $\Cosupph(E)=\varnothing$; showcasing the failure of the cosupport to detect vanishing and consequently, the failure of the colocal-to-global principle. As a result, the local-to-global principle cannot hold either, since it implies the colocal-to-global principle.
\end{ex}

\subsection{Noetherian schemes}

Let $X$ be a noetherian separated scheme and denote by $\rmD(X)$ the derived category of quasi-coherent sheaves over $X$. Then $\rmD(X)$ is a big tt-category whose subcategory of compact objects is $\rmD^{\mathrm{perf}}(X)$ the subcategory of perfect complexes. The Balmer spectrum of $\rmD(X)$ is homeomorphic to the underlying space of $X$~\cite{Thomason97}. The notion of support we consider is the Balmer--Favi support. However, since $\rmD(X)$ satisfies the Telescope Conjecture~\cite{Stevenson13}, one can also choose either the usual homological support or the small smashing support, as they are all identified under the bijections $\mathrm{Spc}^\rms(\rmD(X))\cong \mathrm{Spc}(\rmD^\mathrm{perf}(X))\cong X$.

\begin{thm}
Let $X$ be a noetherian separated scheme. Then $\rmD(X)$ is costratified.
\end{thm}

\begin{proof}
By~\cite[Corollary 8.13]{Stevenson13}, $\rmD(X)$ is stratified. In particular, $\rmD(X)$ satisfies the local-to-global principle. Hence, by~\Cref{cor:ltg-implies-coltg}, $\rmD(X)$ satisfies the colocal-to-global principle. So, it suffices to prove cominimality. Let $\{U_i\}_{i \in I}$ be an open affine cover of $X$. As $X$ is noetherian, any open subset of $X$ is quasi-compact, so its complement is Thomason. The corresponding smashing localization $\rmD(X)(U_i)$ is equivalent to $\rmD(U_i)$. The latter is costratified (in particular it satisfies cominimality) by~\Cref{thm:affine-costrat}. The result follows by~\Cref{cor:comin-covers-compacts}.
\end{proof}

\end{document}